\documentclass[11pt,reqno]{amsart}
\usepackage[utf8]{inputenc} 
\usepackage[T1]{fontenc}    
\usepackage{lmodern}  

\usepackage{graphics,color}
\usepackage{amssymb, amsmath, amsthm, amscd}
\usepackage{latexsym, verbatim, graphicx, amsfonts}
\usepackage[mathscr]{euscript}
\usepackage{amsmath, amsthm, amssymb}
\usepackage{dsfont}
\usepackage{enumitem}
\usepackage{mathrsfs}
\usepackage{mathtools}
\usepackage{float}

\usepackage{xcolor}                 
\usepackage[colorlinks=true, linkcolor=blue, citecolor=black, urlcolor=black]{hyperref}           

\usepackage{aliascnt}
\usepackage[nameinlink,noabbrev]{cleveref}

\theoremstyle{plain}
\newtheorem{theorem}{Theorem}[section]

\newaliascnt{lemma}{theorem}
\newtheorem{lemma}[lemma]{Lemma}
\aliascntresetthe{lemma}
\crefname{lemma}{lemma}{lemmas}
\Crefname{lemma}{Lemma}{Lemmas}

\newaliascnt{proposition}{theorem}
\newtheorem{proposition}[proposition]{Proposition}
\aliascntresetthe{proposition}
\crefname{proposition}{proposition}{propositions}
\Crefname{proposition}{Proposition}{Propositions}

\newaliascnt{corollary}{theorem}
\newtheorem{corollary}[corollary]{Corollary}
\aliascntresetthe{corollary}
\crefname{corollary}{corollary}{corollaries}
\Crefname{corollary}{Corollary}{Corollaries}

\theoremstyle{definition}
\newaliascnt{definition}{theorem}
\newtheorem{definition}[definition]{Definition}
\aliascntresetthe{definition}
\crefname{definition}{definition}{definitions}
\Crefname{definition}{Definition}{Definitions}

\newaliascnt{remark}{theorem}
\newtheorem{remark}[remark]{Remark}
\aliascntresetthe{remark}
\crefname{remark}{remark}{remarks}
\Crefname{remark}{Remark}{Remarks}

\newaliascnt{question}{theorem}
\newtheorem{question}[question]{Question}
\aliascntresetthe{question}
\crefname{question}{question}{questions}
\Crefname{question}{Question}{Questions}

\newaliascnt{example}{theorem}

\aliascntresetthe{example}
\crefname{example}{example}{examples}
\Crefname{example}{Example}{Examples}

\newcommand{\N}{\mathbb{N}}

\newcommand{\norm}[1]{ \| #1 \| }

\DeclareMathOperator{\supp}{supp}

\DeclareMathOperator{\diag}{diag}

\newcommand{\vertiii}[1]{{\left\vert\kern-0.25ex\left\vert\kern-0.25ex\left\vert #1 
    \right\vert\kern-0.25ex\right\vert\kern-0.25ex\right\vert}}

\font\sstext=ecss1000
\font\sssub=ecss1000 at 7pt
\font\sssubsub=ecss1000 at 5pt

\newfam\ssfam
\textfont\ssfam=\sstext
\scriptfont\ssfam=\sssub
\scriptscriptfont\ssfam=\sssubsub

\makeatletter
\def\@tocline#1#2#3#4#5#6#7{\relax
  \ifnum #1>\c@tocdepth
  \else
    \par \addpenalty\@secpenalty\addvspace{#2}%
    \begingroup
      \hyphenpenalty\@M
      \@ifempty{#4}{%
        \@tempdima\csname r@tocindent\number#1\endcsname\relax
      }{%
        \@tempdima#4\relax
      }%
      \parindent\z@
      \leftskip#3\relax
      \advance\leftskip\@tempdima\relax
      \rightskip\@pnumwidth plus4em
      \parfillskip-\@pnumwidth
      #5\leavevmode\hskip-\@tempdima
        \ifcase #1
        \or\or \hskip 1em \or \hskip 2em \else \hskip 3em \fi
        #6\nobreak\relax
      \dotfill\hbox to\@pnumwidth{\@tocpagenum{#7}}\par
      \nobreak
    \endgroup
  \fi
}
\makeatother

\subjclass[2020]{Primary: 46B03; 
Secondary:
46B25, 
46B20. 
}

\keywords{Primary Banach space, $\ell_p$-sums, primary factorisation property, uniform primary factorisation property, maximal ideals of operator algebras}

\begin{document}

\title[Preservation of primariness under vector-valued sums]{Preservation of primariness under $\ell_1$-, $c_0$-, and $\ell_\infty$-sums of Banach spaces}
\author[A. Acuaviva]{Antonio Acuaviva}
\address{School of Mathematical Sciences,
Fylde College,
Lancaster University,
LA1 4YF,
United Kingdom} \email{ahacua@gmail.com}

\date{\today}

\begin{abstract}
    We prove transfer principles for the uniform primary factorisation property (UPFP) from a Banach space $X$ to the vector-valued sequence spaces $\ell_1(X)$, $c_0(X)$ and $\ell_\infty(X)$. The hypotheses are either finite-cotype assumptions on $X$ or $X^*$, or natural self-similarity assumptions on $X$. Consequently, under these conditions, the resulting vector-valued sequence spaces are primary. As applications, we recover the primariness of $\ell_\infty(L_p)$ for $1\leq p<\infty$ without using Bourgain's localisation method, and obtain the primariness of $c_0(L_1)$. We also show that $\ell_1(\Gamma,L_1[0,1])$ has the UPFP for every set $\Gamma$, and consequently that $C[0,1]^*$ has the UPFP and is primary.
\end{abstract}

\maketitle

\tableofcontents

\bigskip
\section{Introduction and organisation}

A Banach space is called \emph{primary} if every decomposition of it as a direct sum has one summand already isomorphic to the whole space. Equivalently, whenever $P\colon X\to X$ is a projection, at least one of the spaces $PX$ and $(I_X-P)X$ is isomorphic to $X$. The primariness problem for classical Banach spaces has played an important role in the isomorphic theory of Banach spaces, and many of its solutions rely on factorisation methods.

A typical feature of such proofs is that they establish a stronger factorisation dichotomy for arbitrary operators. This is captured by the (uniform) primary factorisation property, abbreviated (U)PFP; see \Cref{def:UPFP}. \break Roughly speaking, the PFP asks that, for every operator $T\colon X\to X$, the identity on $X$ factors through either $T$ or $I_X-T$, while the UPFP asks for uniform control of the corresponding factorisation constants. Thus (U)PFP is a natural property to explore in the study of primariness, and is often the more flexible property in permanence arguments. This point of view was used explicitly in the recent work of the present author and Kania on primariness of uncountable $\ell_p$-sums \cite{AcuavivaKania2026}; we refer to that article for a more detailed discussion of primariness and the PFP.

In this paper, we prove transfer principles showing that, under finite-cotype assumptions on $X$ or $X^*$, or under natural self-similarity assumptions on $X$, the UPFP passes from $X$ to the vector-valued sequence spaces $\ell_1(X)$, $c_0(X)$ and $\ell_\infty(X)$. The proofs are surprisingly elementary and recover, by a unified method, several known primariness results for spaces of the form $\ell_\infty(X)$. Our first main results is the following.

\begin{theorem}\label{th: main1}
    Let $X$ be a Banach space with the UPFP. Suppose that at least one of the following holds:
    \begin{enumerate}[label=(\roman*), ref=(\roman*)]
        \item \label{main1:it1} The space $X$ is self-similar, in the sense that $X \simeq \ell_p(X)$ for some $1 \leq p \leq \infty$, or $X \simeq c_0(X)$.
        \item \label{main1:it2} The dual space $X^*$ has finite cotype.
    \end{enumerate}
    Then $\ell_1(X)$ has the UPFP. In particular, $\ell_1(X)$ is primary.
\end{theorem}

We also obtain a similar conclusion in the case of $c_0$ and $\ell_\infty$-sums.

\begin{theorem}\label{th: main2}
    Let $X$ be a Banach space with the UPFP. Suppose that at least one of the following holds:
    \begin{enumerate}[label=(\roman*), ref=(\roman*)]
        \item \label{main2:it1} The space $X$ is self-similar, in the sense that $X \simeq \ell_p(X)$ for some $1 \leq p < \infty$.
        \item \label{main2:it2} The space $X$ has finite cotype.
    \end{enumerate}
    Then both $c_0(X)$ and $\ell_\infty(X)$ have the UPFP. In particular, $c_0(X)$ and $\ell_\infty(X)$ are primary.
\end{theorem}

As an immediate consequence, we recover the primariness of $\ell_\infty(L_p)$ for $1\leq p<\infty$ by elementary methods. The original proofs of Wark \cite{Wark2007} for $1<p<\infty$ and M\"uller \cite{Muller2012} for the full range, rely on substantially more technical methods, in particular Bourgain's localisation method. We also record the case of $c_0(L_1)$, which we could not find explicitly stated in the literature, although it is presumably known.

\begin{corollary}
    The spaces $\ell_\infty(L_p)$, $1\leq p<\infty$, and $c_0(L_1)$ have the UPFP and thus are primary.
\end{corollary}

\begin{proof}
    By \Cref{th: main2}, it is enough to note that $L_p$, $1\leq p<\infty$, has the UPFP and satisfies the self-similarity condition $L_p\simeq \ell_p(L_p)$. For $1<p<\infty$, the UPFP can be extracted from the primariness argument of Alspach, Enflo and Odell \cite{AEO1977}. For $p=1$, the corresponding quantitative statement can be extracted from Capon's argument \cite{Capon1980}, see \Cref{cor:L1-has-UPFP}. It is plausible that a quantitative statement of this kind can also be extracted from Enflo's proof, via Maurey's exposition, of the primariness of $L_p$ \cite{Maurey1975}, although we have not verified this. Therefore \Cref{th: main2} applies and gives the UPFP. Primariness follows directly from the UPFP together with Pe{\l}czy{\'n}ski's decomposition method.
\end{proof}

\begin{remark}
    Lechner and Speckhofer recently proved a much broader factorisation theorem for Haar system Hardy spaces \cite[Theorem 3.1]{LechnerSpeckhofer2025}. In particular, their result also implies the UPFP for $L_p$, $1\leq p<\infty$, and for $\ell_\infty(L_p)$ when $1<p<\infty$. The point of the corollary above is that, once the UPFP of a Banach space X is known, the passage to $\ell_\infty(X)$ follows from the elementary triangular-reduction argument developed here.
\end{remark}

We also record a consequence for spaces of continuous functions on ordinal intervals. Combining our $\ell_1$-sum transfer theorem with the known factorisation properties of $C(\alpha)$, we obtain the remaining $\ell_1$-case in the classification of primariness of the spaces $\ell_p(C(\alpha))$ begun in \cite{Acuaviva2026PrimarinesslpCK}.

\begin{corollary}
    For every ordinal $\alpha$, the space $\ell_1(C(\alpha))$ has the UPFP. In particular, $\ell_1(C(\alpha))$ is primary.
\end{corollary}

\begin{proof}
    Suppose first that $C(\alpha)$ is not isomorphic to $C(\xi\cdot n)$ for any uncountable regular ordinal $\xi$ and any integer $n\geq2$. By the Alspach--Benyamini theorem \cite{AlspachBenyamini1977}, in the form stated in \cite[Proposition~1.9]{Acuaviva2025}, the space $C(\alpha)$ has the PFP; observing that the estimates in the proof are quantitative gives the UPFP. Since the ordinal interval $[0,\alpha]$ is scattered, we have $C(\alpha)^*\simeq \ell_1([0, \alpha])$.
    In particular, $C(\alpha)^*$ has finite cotype and therefore \Cref{th: main1} implies that $\ell_1(C(\alpha))$ has the UPFP.

    Suppose now that $C(\alpha)\simeq C(\xi\cdot n)$ for some uncountable regular ordinal $\xi$ and some integer $n\geq2$. Since $C(\xi\cdot n)$ is isomorphic to the direct sum of $n$ copies of $C(\xi)$, we have $\ell_1(C(\alpha))\simeq \ell_1(C(\xi))$. By the first part of the proof, applied to $\xi$, and by the isomorphic invariance of the UPFP, $\ell_1(C(\alpha))$ has the UPFP. The primariness follows from the UPFP together with Pe{\l}czy{\'n}ski's decomposition method, since $\ell_1(C(\alpha))\simeq \ell_1(\ell_1(C(\alpha)))$.
\end{proof}
Finally, although the transfer principles in \Cref{th: main1,th: main2} apply only under the stated abstract hypotheses on the summand $X$, the diagonal-reduction method can sometimes be implemented beyond those hypotheses by using additional structure of $X$. We illustrate this phenomenon for the summand $L_1[0,1]$, obtaining the following result. It generalises a previous result of the author and Kania \cite{AcuavivaKania2026}, which assumed the negation of $\mathrm{CH}$.

\begin{theorem}\label{th:main3}
   For every set $\Gamma$, the space $\ell_1(\Gamma,L_1[0,1])$ has the UPFP. Consequently, $C[0,1]^*$ has the UPFP and, in particular, is primary.
\end{theorem}

\subsection{Organisation and proof strategy}

In \Cref{sec:notation}, we introduce the notation and basic facts used throughout the paper, including the matrix representation of operators on vector-valued sequence spaces. Once this framework is in place, the proofs of \Cref{th: main1,th: main2} are fairly straightforward, we present them in \Cref{sec:diagonal-and-proofs}. The main idea is to reduce, up to a small perturbation, an operator $T\colon E(X)\to E(X)$ to diagonal form. We then use the UPFP of $X$ to factor the identity through either the resulting diagonal operator $D$ or through $I_{E(X)}-D$; see \Cref{prop:diagonal-dichotomy}. Finally, we use the fact that factorisation of the identity is stable under sufficiently small perturbations; see \Cref{lmm:factorisation-perturbation}.

The diagonal reduction is obtained in two steps: an upper-triangular reduction, proved in \Cref{sec:upper-triangular}, and a lower-triangular reduction, proved in \Cref{sec:lower-triangular}. There are two types of hypotheses. In the finite-cotype case, the relevant finite-dimensional geometry is sufficiently far from that of the scalar sequence spaces involved: finite cotype of $X^*$ is used for $\ell_1(X)$, while finite cotype of $X$ is used for $c_0(X)$ and $\ell_\infty(X)$. In the self-similar case, the reductions are obtained by contrasting incompatible sequence-space geometries: for $\ell_1$-sums, the internal $\ell_p$, $1<p\leq\infty$, or $c_0$ structure is contrasted with the external $\ell_1$ structure, while for $c_0$- and $\ell_\infty$-sums, the internal $\ell_p$ structure is contrasted with the external $\ell_\infty$ norm structure.

In \Cref{sec:uncountable-l1-sum}, we apply a variant of the same diagonal method to uncountable $\ell_1$-sums with summand $L_1[0,1]$. The additional ingredients are a finite-coordinate localisation lemma for operators into $\ell_1(\Gamma,L_1[0,1])$ and a finite free-set theorem of Hajnal. This gives the UPFP for $\ell_1(\Gamma,L_1[0,1])$ for every set $\Gamma$, and hence the UPFP and primariness of $C[0,1]^*$.

Finally, \Cref{sec:remarks-and-questions} contains some additional remarks concerning the relationship between the PFP and UPFP of $X$ and those of $\ell_p(X)$, $c_0(X)$ and $\ell_\infty(X)$, together with some open questions.

\bigskip
\section{Notation and preliminary results}\label{sec:notation}

We use standard notation and conventions, unless explicitly stated other\-wise. The scalar field is denoted by $\mathbb{K}$ and may be either $\mathbb{R}$ or $\mathbb{C}$. By an \emph{operator} we mean a bounded linear map between Banach spaces. If $X$ is a Banach space, then $I_X \colon X \to X$ denotes the identity operator on $X$, $B_X$ denotes its closed unit ball, and $\mathscr{B}(X)$ denotes the space of operators on $X$. If $X$ has cotype $q$, we write $C_q(X)$ for its cotype $q$ constant. Further notation will be introduced as needed. We shall use the following terminology.

\begin{definition}
    Let $T \colon X \to Y$ and $S \colon Z \to W$ be operators between Banach spaces. We say that $S$ \emph{factors through} $T$ with constant $C$ if there are operators $V \colon Z \to X$ and $U \colon Y \to W$ such that
    \begin{equation*}
    UTV = S
    \quad \text{and} \quad
    \norm{U}\norm{V} \leq C.
    \end{equation*}
    If such a constant $C \geq 1$ exists, we simply say that $S$ factors through $T$.
\end{definition}

The uniform primary factorisation property will play a central role throughout the paper. We recall its definition.

\begin{definition}\label{def:UPFP}
    Let $X$ be a Banach space.
    \begin{enumerate}[label = (\alph*), ref = (\alph*)]
    \item We say that $X$ has the \emph{primary factorisation property} (PFP) if, for every operator $T \colon X \to X$, the identity operator $I_X$ factors through either $T$ or $I_X - T$.
    \item Let $C \geq 1$. We say that $X$ has the \emph{$C$-primary factorisation property} ($C$-PFP) if, for every operator $T \colon X \to X$, the identity operator $I_X$ factors through either $T$ or $I_X - T$ with constant $C$.
    \item We say that $X$ has the \emph{uniform primary factorisation property} (UPFP) if $X$ has the $C$-PFP for some $C \geq 1$.
    \end{enumerate}
\end{definition}

Throughout the paper, we shall use without further comment that the UPFP is invariant under isomorphisms; see, for instance, \cite[Proposition~2.5]{AcuavivaKania2026}.

\subsection{Vector-valued sequence spaces}

Given a Banach space $X$, we shall often consider vector-valued sequence spaces. Let $E$ be a Banach space with a fixed normalized $1$-unconditional basis $(e_n)_{n\in\N}$. We define $E(X)$ to be the Banach space of all sequences $x=(x_n)_{n\in\N}$ in $X$ such that
\begin{equation*}
    \sum_{n=1}^{\infty}\norm{x_n}e_n
\end{equation*}
norm converges in $E$, equipped with the norm
\begin{equation*}
    \norm{x}=\left\|\sum_{n=1}^{\infty}\norm{x_n}e_n\right\|.
\end{equation*}
This construction depends on the chosen unconditional basis of $E$, which will be fixed throughout and left implicit in the notation.

For $x=(x_n)_{n\in\N}\in E(X)$, we define the \emph{support} of $x$ by
\begin{equation*}
    \supp(x)=\{n\in\N:x_n\neq 0\}.
\end{equation*}
If $M\subseteq\N$, we set
\begin{equation*}
    E(M,X)=\{x\in E(X):\supp(x)\subseteq M\}.
\end{equation*}
We denote by
\begin{equation*}
    J_M\colon E(M,X)\to E(X) \quad \text{and} \quad P_M\colon E(X)\to E(M,X)
\end{equation*}
the canonical inclusion and projection, respectively. Thus $P_MJ_M=I_{E(M,X)}$. If $M=\{n\}$ is a singleton, we write $J_n$ and $P_n$ in place of $J_{\{n\}}$ and $P_{\{n\}}$; under the natural identification of $E(\{n\},X)$ with $X$, these are precisely the canonical coordinate embedding and projection.

If $E$ and $F$ are Banach spaces with normalized $1$-unconditional bases, and $T\colon E(X)\to F(Y)$ is an operator, then for each $n,m\in\N$ we set
\begin{equation*}
    T_{n,m}=P_nTJ_m\in\mathscr{B}(X,Y).
\end{equation*}
We call $(T_{n,m})_{n,m\in\N}$ the \emph{matrix of} of $T$. Since the finitely supported vectors are dense in $E(X)$, this matrix determines $T$. More precisely, for every $x=(x_m)_{m\in\N}\in E(X)$ and every $n\in\N$, we have
\begin{equation*}
    P_nTx=\sum_{m=1}^{\infty}T_{n,m}x_m,
\end{equation*}
where the series converges in $Y$. If $(T_n)_{n\in\N}$ is a uniformly bounded sequence in $\mathscr{B}(X)$, then the formula
\begin{equation*}
    \diag(T_n:n\in\N)(x_n)_{n\in\N}=(T_nx_n)_{n\in\N}
\end{equation*}
defines an operator on $E(X)$, called a \emph{diagonal operator}.

\begin{definition}
    Let $T \colon E(X) \to E(X)$ be an operator with matrix of $(T_{n,m})_{n,m\in\N}$. We define the \emph{diagonal part} of $T$ to be the diagonal operator $D\colon E(X)\to E(X)$ given by
    \begin{equation*}
        D(x_n)_{n\in\N}=(T_{n,n}x_n)_{n\in\N}.
    \end{equation*}
    Equivalently, $D=\diag(T_{n,n}:n\in\N)$. 
\end{definition}

Observe that since $\sup_{n\in\N}\norm{T_{n,n}}\leq\norm{T}$, the diagonal part is indeed always an operator.

Let $c_{00}(X)$ denote the subspace of $E(X)$ consisting of finitely supported vectors.

\begin{definition}
    Let $T \colon E(X) \to E(X)$ be an operator with matrix of $(T_{n,m})_{n,m \in \N}$. We define the \emph{upper triangular part} and the \emph{lower triangular part} of $T$ to be the linear maps initially defined on $c_{00}(X)$ by the matrices
    \begin{equation*}
        U_{n,m}=\begin{cases} T_{n,m}, & n\leq m, \\ 0, & n>m, \end{cases}
        \qquad
        L_{n,m}=\begin{cases} T_{n,m}, & n\geq m, \\ 0, & n<m. \end{cases}
    \end{equation*}
\end{definition}

These maps are well defined on $c_{00}(X)$ and take values in $E(X)$. They need not, in general, extend to operators on $E(X)$. Since $c_{00}(X)$ is dense in $E(X)$, when one of them is bounded on $c_{00}(X)$, we use the same notation for its unique bounded extension to an operator on $E(X)$; in that case the other one also extends boundedly, since on $c_{00}(X)$ we have
\begin{equation*}
    T=U+L-D,
\end{equation*}
where $D$ denotes the diagonal part of $T$. Analogous considerations apply to the strictly upper and strictly lower triangular parts, obtained by keeping only the entries with $n<m$ and $n>m$, respectively.

The same terminology applies to operators on $E(M,X)$, with the order inherited from $\N$. Thus, if $M=\{m_j:j\in\N\}$ is written in increasing order, the upper triangular part keeps only the entries $P_{m_i}TJ_{m_j}$ with $i\leq j$, while the lower triangular part keeps only those with $i\geq j$.

\begin{remark}
    Throughout the paper, we leave the extension from $c_{00}(X)$ to $E(X)$ implicit and focus on the only problematic point, namely the boundedness of the corresponding triangular truncation. Equivalently, one may first regard these truncations as coordinatewise defined linear maps from $E(X)$ into $\prod_{n\in\N}X$. In each case where such a truncation is used, the estimates show that it actually takes values in $E(X)$ and, in fact, defines a bounded operator on $E(X)$.
\end{remark}

We shall also use the notation $\ell_\infty(X)$ in the usual sense, with norm
\begin{equation*}
    \norm{(x_n)_{n\in\N}}_{\ell_\infty(X)}=\sup_{n\in\N}\norm{x_n}.
\end{equation*}
For $\ell_\infty(X)$ we keep the notation $P_n,J_n,P_M,J_M$ and $T_{n,m}=P_nTJ_m$. However, one must be more careful: the matrix $(T_{n,m})_{n,m\in\N}$ does not, in general, determine the operator $T\in\mathscr{B}(\ell_\infty(X))$; see \cite[Example 2.3]{Acuaviva2026PrimarinesslpCK} for an easy example.

\begin{definition}
    Let $T\colon \ell_\infty(X) \to \ell_\infty(X)$ be an operator with operator matrix $(T_{n,m})_{n,m \in \N}$. We say that $T$ \emph{acts according to its matrix} if, for every $x=(x_n)_{n\in\N}\in\ell_\infty(X)$ and every $m\in\N$, the series
    \begin{equation*}
        \sum_{n=1}^{\infty}T_{m,n}x_n
    \end{equation*}
    norm converges in $X$ and
    \begin{equation*}
        P_mTx=\sum_{n=1}^{\infty}T_{m,n}x_n.
    \end{equation*}
\end{definition}

The same coordinatewise definition of diagonal operators will be used on $\ell_\infty(X)$. Thus, if $(T_n)_{n\in\N}$ is a uniformly bounded sequence in $\mathscr{B}(X)$, then
\begin{equation*}
    \diag(T_n:n\in\N)(x_n)_{n\in\N}=(T_nx_n)_{n\in\N}
\end{equation*}
defines an operator on $\ell_\infty(X)$, which acts according to its matrix.

We have a similar notion of diagonal and lower triangular parts for operators on $\ell_\infty(X)$. In this case these parts are genuine operators and act according to their matrices in the sense just defined.

\begin{definition}
    Let $T\colon\ell_\infty(X)\to\ell_\infty(X)$ be an operator, and let $(T_{n,m})_{n,m\in\N}$ be its matrix of. We define the \emph{diagonal part} $D$ of $T$ to be the operator $D\colon \ell_\infty(X)\to \ell_\infty(X)$ defined, for each $x=(x_n)_{n\in\N}\in\ell_\infty(X)$ and each $m\in\N$, by
    \begin{equation*}
        P_mDx=T_{m,m}x_m.
    \end{equation*}
    Similarly, we define the \emph{lower triangular part} $L$ of $T$ to be the operator $L\colon \ell_\infty(X)\to \ell_\infty(X)$ defined, for each $x=(x_n)_{n\in\N}\in\ell_\infty(X)$ and each $m\in\N$, by
    \begin{equation*}
        P_mLx=\sum_{n=1}^{m}T_{m,n}x_n.
    \end{equation*}
\end{definition}

Indeed, these formulae define operators, with $\norm{D}\leq\norm{T}$ and $\norm{L}\leq\norm{T}$. For the latter, observe that, for every $x\in\ell_\infty(X)$ and every $m\in\N$,
\begin{equation*}
    P_mLx=P_mTJ_{\{1,\ldots,m\}}P_{\{1,\ldots,m\}}x.
\end{equation*}

Moreover, both $D$ and $L$ act according to their matrices in the sense defined above. We shall not use the upper triangular part of an arbitrary operator on $\ell_\infty(X)$.

\subsection{Iterated vector-valued sequence spaces} When working with iterated sequence spaces $\ell_1(E(X))$ or $E(\ell_p(X))$, we distinguish the outer and inner coordinates. The coordinate projections and embeddings associated with the outer sequence space will be denoted by
\begin{equation*}
    \widehat{P}_{n},\widehat{P}_{M},\widehat{J}_{n},\widehat{J}_{M},
\end{equation*}
whereas the unadorned symbols $P_n,P_M,J_n,J_M$ refer to the inner sequence space. Thus, for example, if $T\colon \ell_1(E(X))\to\ell_1(E(X))$, then $\widehat{P}_{m}T\widehat{J}_{n}$ is an operator from $E(X)$ into $E(X)$.

Let $E$ and $F$ be Banach spaces belonging to the family $\{c_0\} \cup \{\ell_p: 1 \leq p \leq \infty\}$. If $\mathbf{n}=(n_j)_{j\in\N}$ and $\mathbf{m}=(m_j)_{j\in\N}$ are increasing sequences of positive integers, we define
\begin{equation*}
    J_{\mathbf{n},\mathbf{m}}\colon F(X)\to F(E(X))
\end{equation*}
as follows. Given $x=(x_j)_{j\in\N}\in F(X)$, the vector $J_{\mathbf{n},\mathbf{m}}x$ has outer coordinate $J_{m_j}x_j$ in position $n_j$, and zero in all other outer coordinates. That is,
\begin{equation*}
    \widehat{P}_{n_j}J_{\mathbf{n},\mathbf{m}}x=J_{m_j}x_j
\end{equation*}
for every $j\in\N$, while $\widehat{P}_{n}J_{\mathbf{n},\mathbf{m}}x=0$ whenever $n\notin\{n_j:j\in\N\}$.

We also define
\begin{equation*}
    P_{\mathbf{n},\mathbf{m}}\colon F(E(X))\to F(X)
\end{equation*}
by extracting the corresponding coordinates. Thus, if $y=(y_n)_{n\in\N}\in F(E(X))$, then
\begin{equation*}
    P_{\mathbf{n},\mathbf{m}}y=(P_{m_j}y_{n_j})_{j\in\N}.
\end{equation*}
Equivalently,
\begin{equation*}
    P_{\mathbf{n},\mathbf{m}}y=(P_{m_j}\widehat{P}_{n_j}y)_{j\in\N}.
\end{equation*}
\bigskip
\section{Upper triangular reduction}\label{sec:upper-triangular}

In this section we prove the upper-triangular reductions needed later for the diagonal reductions. The general mechanism is the same throughout: first one proves a forward reduction lemma, which allows us, after passing to suitable infinite subsets, to choose a coordinate whose image has arbitrarily small projection onto future coordinates; then a standard recursive selection turns these one-step estimates into an upper-triangular form.

We begin with the case of $\ell_1(X)$. The finite dual-cotype case is recorded in \Cref{prop:upper-triangular-ell1-dual-cotype}, while the self-similar case is obtained from the forward reduction in \Cref{lmm:forward-reduction-from-reprod-ell1} and the inductive construction in \Cref{prop:upper-triangular-reduction-from-reprod-ell1}. We then treat the $c_0(X)$ and $\ell_\infty(X)$ cases: finite cotype is handled in \Cref{lmm:cotype-forward-reduction} and \Cref{prop:upper-triangular-reduction-from-cotype-ellinfty}, and self-similarity in \Cref{lmm:forward-reduction-from-reprod-ellinfty} and \Cref{prop:upper-triangular-reduction-from-reprod-ellinfty}.

\subsection{The \texorpdfstring{$\ell_1$}{ell_1} case}

\subsubsection{Finite cotype of the dual space}

We first consider operators on $\ell_1(X)$ under the assumption that $X^*$ has finite cotype. The required upper-triangular reduction follows from the same compactness mechanism used in \cite{Acuaviva2026PrimarinesslpCK}; for completeness, we record the precise form needed here.

\begin{proposition}[Reduction to upper-triangular operators from dual cotype of the dual space]\label{prop:upper-triangular-ell1-dual-cotype}
    Let $X$ be a Banach space such that $X^*$ has finite cotype. Then, for every operator $T \colon \ell_1(X) \to \ell_1(X)$ and every $\varepsilon > 0$, there exists an infinite subset $M \subseteq \N$ such that
    \begin{equation*}
        \norm{P_M T J_M - U} < \varepsilon,
    \end{equation*}
    where $U \colon \ell_1(M, X) \to \ell_1(M, X)$ is the upper triangular part of $P_M T J_M$.
\end{proposition}
\begin{proof}
    Since $X^*$ has finite cotype, it contains no copy of $c_0$. Hence, by a classical theorem of Pe{\l}czy\'nski \cite{pelczynski1965strictly}, every operator from $\ell_\infty = C(\beta\mathbb{N})$ into $X^*$ is weakly compact. Therefore $S^*\colon \ell_\infty\to X^*$ is weakly compact for any $S\colon X\to\ell_1$. By Gantmacher's theorem, $S$ is weakly compact. Since $\ell_1$ has the Schur property, weakly compact subsets of $\ell_1$ are norm compact. Hence $S$ is compact. The result is now a particular instance of \cite[Proposition 3.5]{Acuaviva2026PrimarinesslpCK}.
\end{proof}

\subsubsection{Self-similarity of the space}

Now we move to the case when $X \simeq \ell_p(X)$ for some $1 < p \leq \infty$ or $X \simeq c_0(X)$. Before we can prove the upper-triangular reduction, we will need the following lemma.

\begin{lemma}[Forward reduction from self-similarity]\label{lmm:forward-reduction-from-reprod-ell1}
    Let $E = \ell_p$ for $1 < p \leq \infty$ or $E = c_0$, $X$ and $Y$ be Banach spaces and $T\colon E(X) \to \ell_1(Y)$ be an operator. Then, for every $\varepsilon > 0$, there exist a pair of infinite sets $A, M \subseteq \N$ such that $\norm{P_M T J_A}<\varepsilon$.
\end{lemma}

\begin{proof}
    We first do the case $E = \ell_p$ for $1 < p < \infty$. Proceed by contradiction and suppose that the statement is false. Then there is $\varepsilon>0$ such that, for every pair of infinite sets $A, M\subseteq\N$, we have $\norm{P_MTJ_A}\geq \varepsilon$. Choose $N \in \N$ such that $N^{(p-1)/p} \varepsilon > 2\norm{T}$, and partition $\N$ into infinite pairwise disjoint sets $A_1, \dots, A_N$ and $M_1,\dots,M_N$ respectively.

    By assumption, for each $n=1,\dots,N$ we have $\norm{P_{M_n}TJ_{A_n}} \geq \varepsilon$. Hence, we may choose $x_n \in B_{\ell_p(A_n, X)}$ and $y^*_n\in B_{\ell_1(M_n,Y)^*}$ such that 
    \begin{equation*}
        |y^*_nP_{M_n} TJ_{A_n} x_n|> \varepsilon/2.
    \end{equation*}

    For each choice of signs $\theta = (\theta_1,\dots,\theta_N) \in \{-1,1\}^N$, set
    \begin{equation*}
        x_\theta=\sum_{n=1}^N \theta_{n} J_{A_n}x_{n}.
    \end{equation*}
    Since the vectors $J_{A_n}x_{n}$ have disjoint supports in $\ell_p(X)$, we have $\norm{x_\theta}\leq N^{1/p}$. Thus $\norm{Tx_\theta}\leq \norm{T} N^{1/p}$ for every choice of signs, and averaging gives
    \begin{equation*}
        \norm{T} N^{1/p} \geq \mathbb{E}_\theta [\norm{Tx_\theta}].
    \end{equation*}
    Since the sets $M_1,\dots,M_N$ form a partition of $\N$, we have
    \begin{equation*}
        \mathbb{E}_\theta [\norm{Tx_\theta}] = \sum_{n=1}^N \mathbb{E}_\theta [ \norm{P_{M_n}Tx_\theta}] \geq \sum_{n=1}^N |\mathbb{E}_\theta [\theta_n y_n^*P_{M_n}Tx_\theta]|.
    \end{equation*}
    For each fixed $n$, since $\mathbb{E}_\theta[\theta_n \theta_m] = 0$ if $n \not = m$ and $\mathbb{E}_\theta[\theta_n^2] = 1$, we have
    \begin{equation*}
        |\mathbb{E}_\theta [\theta_n y_n^*P_{M_n}Tx_\theta]| = |y^*_nP_{M_n} TJ_{A_n} x_n|>\varepsilon/2.
    \end{equation*}
    Therefore
    \begin{equation*}
        \norm{T} N^{1/p}  \geq \mathbb{E}_\theta [\norm{Tx_\theta}] > N \varepsilon/2,
    \end{equation*}
    contradicting the choice of $N$. This contradiction proves the result.

    The proofs for the cases $E = \ell_\infty$ and $E = c_0$ are identical, now choosing $N$ so that $N \varepsilon > 2 \norm{T}$ and observing that $\norm{x_\theta} \leq 1$ in this case.
\end{proof}

We are now ready to prove an upper-triangular reduction in the case of $\ell_p$ or $c_0$ self-similarity.

\begin{proposition}[Reduction to upper-triangular operators from self-similarity]\label{prop:upper-triangular-reduction-from-reprod-ell1}
    Let $E = \ell_p$ for $1 < p \leq \infty$ or $E = c_0$, $X$ be a Banach space and $T\colon \ell_1(E(X))\to \ell_1(E(X))$ be an operator. Then, for every $\varepsilon > 0$, there exist an increasing sequence $\mathbf{n}=(n_j)_{j\in\N}$ and infinite sets $(A_j)_{j \in \N}$ so that for all $j \in \N$ we have
    \begin{equation*}
        \norm{\widehat{P}_{M_j} T \widehat{J}_{n_j} J_{A_j}} < \varepsilon/2^{j+1},
    \end{equation*}
    where $M_j = \{n_i: i > j \}$.
\end{proposition}

\begin{proof}
    We recursively construct an increasing sequence $(n_j)_{j\in\N}$ and infinite subsets $(A_j)_{j\in\N}$ and $(\widetilde M_j)_{j\in\N}$ of $\N$ such that, for every $j\in\N$,
    \begin{equation*}
        n_j\in \widetilde M_{j-1}, \qquad \widetilde M_j\subseteq \widetilde M_{j-1}\cap\{n\in\N:n>n_j\},
    \end{equation*}
    and
    \begin{equation*}
        \norm{\widehat{P}_{\widetilde M_j}T\widehat{J}_{n_j}J_{A_j}} < \varepsilon/2^{j+1},
    \end{equation*}
    where $\widetilde M_0=\N$.

    Suppose that the construction has been carried out to stage $j-1$, in particular $\widetilde M_{j-1}$ has been constructed. Choose $n_j\in \widetilde M_{j-1}$, and set
    \begin{equation*}
        L_j=\widetilde M_{j-1}\cap\{n\in\N:n>n_j\}.
    \end{equation*}
    Then $L_j$ is infinite. Applying \Cref{lmm:forward-reduction-from-reprod-ell1} to the operator
    \begin{equation*}
        \widehat{P}_{L_j}T\widehat{J}_{n_j}\colon E(X)\to \ell_1(L_j, E(X)),
    \end{equation*}
    after identifying $\ell_1(L_j,E(X))$ with $\ell_1(E(X))$, we obtain infinite sets $A_j\subseteq\N$ and $\widetilde M_j\subseteq L_j$ such that
    \begin{equation*}
        \norm{\widehat{P}_{\widetilde M_j}T\widehat{J}_{n_j}J_{A_j}} < \varepsilon/2^{j+1}.
    \end{equation*}
    This completes the recursive construction.

    By construction, $(n_j)_{j\in\N}$ is strictly increasing. Moreover, if
    \begin{equation*}
        M_j=\{n_i:i>j\},
    \end{equation*}
    then $M_j\subseteq \widetilde M_j$ for every $j\in\N$. Hence
    \begin{equation*}
        \widehat{P}_{M_j}T\widehat{J}_{n_j}J_{A_j} = \widehat{P}_{M_j}\widehat{P}_{\widetilde M_j}T\widehat{J}_{n_j}J_{A_j},
    \end{equation*}
    and therefore
    \begin{equation*}
        \norm{\widehat{P}_{M_j}T\widehat{J}_{n_j}J_{A_j}} \leq \norm{\widehat{P}_{\widetilde M_j}T\widehat{J}_{n_j}J_{A_j}} < \varepsilon/2^{j+1}.
    \end{equation*}
    This proves the result.
\end{proof}

\subsection{The \texorpdfstring{$\ell_\infty$ and $c_0$}{ell-infinity and c0} cases}

\subsubsection{Finite cotype}

We consider now the upper-triangular reduction in the case of $c_0$ and $\ell_\infty$ sums, we start with the case of finite cotype. We shall need the following lemma.

\begin{lemma}[Forward reduction from cotype]\label{lmm:cotype-forward-reduction}
    Let $X$ be a Banach space with cotype $q<\infty$, $E$ denote either $c_0$ or $\ell_\infty$, $\varepsilon>0$, and $T\colon \ell_\infty^N(X)\to E(X)$ be an operator. Denote by $C_q(X)$ the cotype constant of X and suppose that
    \begin{equation*}
        N^{1/q}\varepsilon> 2 C_q(X)\norm{T}.
    \end{equation*}
    Then there exist $1\leq n\leq N$ and an infinite subset $M\subseteq \mathbb{N}$ such that
    \begin{equation*}
        \norm{P_m T J_n}<\varepsilon \qquad \text{for all } m\in M.
    \end{equation*}
\end{lemma}

\begin{proof}
    We proceed by contradiction and assume that the statement is false. Then, for every $1\leq n\leq N$, the set
    \begin{equation*}
        \{m\in \mathbb{N}:\norm{P_m T J_n} < \varepsilon\}
    \end{equation*}
    is finite. Hence we may choose $m\in \mathbb{N}$ such that for all $1 \leq n \leq N$ we have $\norm{P_m T J_n}\geq\varepsilon$. For each $n$, choose $x_n\in B_X$ such that
    \begin{equation*}
        \norm{P_mTJ_nx_n}>\varepsilon/2.
    \end{equation*}
    For each choice of signs $\theta=(\theta_1,\ldots,\theta_N)\in\{-1,1\}^N$, put
    \begin{equation*}
        x_\theta=\sum_{n=1}^N\theta_nJ_nx_n.
    \end{equation*}
    Naturally $\norm{x_\theta}_{\ell_\infty^N(X)}\leq 1$ so that
    \begin{equation*}
        \norm{T}^q\geq \mathbb{E}_\theta[\norm{Tx_\theta}^q]\geq \mathbb{E}_\theta[\norm{P_mTx_\theta}^q].
    \end{equation*}
    Using the cotype $q$ estimate for $X$ we obtain
    \begin{align*}
        \mathbb{E}_\theta [\norm{P_mTx_\theta}^q] &=\mathbb{E}_\theta \left[\left\|\sum_{n=1}^N\theta_nP_mTJ_nx_n\right\|^q \right]\geq C_q(X)^{-q}\sum_{n=1}^N\norm{P_mTJ_nx_n}^q \\
        &> C_q(X)^{-q}N(\varepsilon/2)^q.
    \end{align*}
    Thus
    \begin{equation*}
        N^{1/q}\varepsilon<2C_q(X)\norm{T},
    \end{equation*}
    contradicting the hypothesis. This contradiction proves the result.
\end{proof}

\begin{remark}
    The point of the previous lemma is not merely the coordinatewise conclusion, but the fact that for $E=c_0$ or $E=\ell_\infty$ one has
    \begin{equation*}
        \norm{P_M T J_n}
        =
        \sup_{m\in M}\norm{P_m T J_n}.
    \end{equation*}
    Thus, coordinatewise smallness on $M$ gives smallness of $P_MTJ_n$. Similar coordinatewise forward reductions can be proved for other $\ell_p$-sums, but they do not by themselves give the required estimate for $P_MTJ_n$, since the target norm is no longer a supremum norm.
\end{remark}

We can now give the upper triangular reduction in the case of finite cotype.

\begin{proposition}[Reduction to upper-triangular operators from cotype]\label{prop:upper-triangular-reduction-from-cotype-ellinfty}
    Let $E=c_0$ or $E=\ell_\infty$, $X$ be a Banach space with finite cotype, and $T\colon E(X)\to E(X)$ be an operator. Then, for every $\varepsilon>0$, there exists an infinite subset $M\subseteq\N$ such that, for every $m\in M$,
    \begin{equation*}
        \sum_{\substack{n<m\\ n\in M}}\norm{P_mTJ_n}<\varepsilon.
    \end{equation*}
\end{proposition}

\begin{proof}
    Let $q<\infty$ be such that $X$ has cotype $q$, and let $C_q(X)$ denote the cotype $q$ constant of $X$. For each $j\in\N$, choose $N_j\in\N$ such that
    \begin{equation*}
        N_j^{1/q}\varepsilon/2^{j+1}>2C_q(X)\norm{T}.
    \end{equation*}

    Using \Cref{lmm:cotype-forward-reduction}, we recursively construct an increasing sequence $(n_j)_{j\in\N}$ and a decreasing sequence $(M_j)_{j\in\N_0}$ of infinite subsets of $\N$, with $M_0=\N$, such that $n_j\in M_{j-1}$, $M_j\subseteq M_{j-1}\cap\{m\in\N:m>n_j\}$, and for all $m \in M_j$ we have
    \begin{equation*}
        \norm{P_mTJ_{n_j}}<\varepsilon/2^{j+1}.
    \end{equation*}

    Suppose that $M_{j-1}$ has been constructed. Choose distinct elements $a_1<\ldots<a_{N_j}$ of $M_{j-1}$, and set
    \begin{equation*}
        L_j=M_{j-1}\cap\{m\in\N:m>a_{N_j}\}.
    \end{equation*}
    Put $A=\{a_1,\ldots,a_{N_j}\}$, and define $J_A\colon \ell_\infty^{N_j}(X)\to E(X)$ by
    \begin{equation*}
        J_A(x_1,\ldots,x_{N_j})=\sum_{s=1}^{N_j}J_{a_s}x_s.
    \end{equation*}
    Apply \Cref{lmm:cotype-forward-reduction} to the operator
    \begin{equation*}
        P_{L_j}TJ_A\colon \ell_\infty^{N_j}(X)\to E(L_j,X),
    \end{equation*}
    where we identify $E(L_j,X)$ with $E(X)$. Since $\norm{P_{L_j}TJ_A}\leq\norm{T}$, we obtain $1\leq s_j\leq N_j$ and an infinite set $M_j\subseteq L_j$ such that, for all $m\in M_j$, we have
    \begin{equation*}
        \norm{P_mTJ_{a_{s_j}}}<\varepsilon/2^{j+1}.
    \end{equation*}
    Set $n_j=a_{s_j}$. This completes the recursive construction.

    Let $M=\{n_j:j\in\N\}$. Fix $j\in\N$. If $i<j$, then $n_j\in M_i$, and therefore
    \begin{equation*}
        \norm{P_{n_j}TJ_{n_i}}<\varepsilon/2^{i+1}.
    \end{equation*}
    Hence
    \begin{equation*}
        \sum_{\substack{n<n_j\\ n\in M}}\norm{P_{n_j}TJ_n} = \sum_{i<j}\norm{P_{n_j}TJ_{n_i}} < \sum_{i<j}\varepsilon/2^{i+1} < \varepsilon.
    \end{equation*}
    Since $j\in\N$ was arbitrary, the proof is complete.
\end{proof}

The conclusion of \Cref{prop:upper-triangular-reduction-from-cotype-ellinfty} should be compared with the upper-triangular reduction obtained for $\ell_1(X)$ in \Cref{prop:upper-triangular-ell1-dual-cotype}. There, one obtains norm approximation of a suitable compression by its upper triangular part. In the present setting, and especially for $\ell_\infty(X)$, matrix representations require more care; for this reason we only record the weaker entry-wise estimate above. This is the form of upper-triangular reduction that will be needed.

\subsubsection{Self-similarity of the space}

We now argue that we can get a similar upper-triangular reduction in the case of $\ell_p$-self-similarity. We need the following preliminary result.

\begin{lemma}[Forward reduction from self-similarity]\label{lmm:forward-reduction-from-reprod-ellinfty}
    Let $1 \leq p < \infty$, $E = c_0$ or $E = \ell_\infty$, $X$ and $Y$ be Banach spaces, and $T\colon \ell_\infty^N(Y) \to E(\ell_p(X))$ be an operator. Suppose that
    \begin{equation*}
        N^{1/p}\varepsilon>2\norm{T}.
    \end{equation*}
    Then there exist $1\leq n\leq N$, an infinite set $M\subseteq\N$, and, for each $m\in M$, an infinite set $A_m\subseteq\N$ such that
    \begin{equation*}
        \norm{P_{A_m}\widehat{P}_mTJ_n}<\varepsilon \qquad \text{for all } m\in M.
    \end{equation*}
\end{lemma}

\begin{proof}
    We proceed by contradiction and suppose that the conclusion fails. For each $1\leq n\leq N$, set
    \begin{equation*}
        M(n)=\{m\in\N: \norm{P_A\widehat{P}_mTJ_n}<\varepsilon \text{ for some infinite } A\subseteq\N\}.
    \end{equation*}
    By assumption each $M(n)$ is finite so that we may choose $m\in\N\setminus\bigcup_{n=1}^N M(n)$. By the choice of $m$, for every $1\leq n\leq N$ and every infinite set $A\subseteq\N$ we have
    \begin{equation*}
        \norm{P_A\widehat{P}_mTJ_n}\geq\varepsilon.
    \end{equation*}

    Let $B_1,\ldots,B_N$ be a partition of $\N$ into infinite sets. For each $1\leq n\leq N$, choose $y_n\in B_Y$ and $x_n^*\in B_{\ell_p(B_n,X)^*}$ such that
    \begin{equation*}
        |x_n^*P_{B_n}\widehat{P}_mTJ_ny_n|>\varepsilon/2.
    \end{equation*}
    For each choice of signs $\theta=(\theta_1,\ldots,\theta_N)\in\{-1,1\}^N$, put
    \begin{equation*}
        y_\theta=\sum_{n=1}^N\theta_nJ_ny_n.
    \end{equation*}
    Since the domain is $\ell_\infty^N(Y)$, we have $\norm{y_\theta}\leq 1$. Therefore
    \begin{equation*}
        \norm{T}^p\geq \mathbb{E}_\theta[\norm{\widehat{P}_mTy_\theta}_{\ell_p(X)}^p].
    \end{equation*}
    Since the sets $B_1,\ldots,B_N$ are disjoint, we obtain
    \begin{equation*}
        \mathbb{E}_\theta[\norm{\widehat{P}_mTy_\theta}_{\ell_p(X)}^p]\geq \sum_{n=1}^N\mathbb{E}_\theta[\norm{P_{B_n}\widehat{P}_mTy_\theta}^p]\geq \sum_{n=1}^N\mathbb{E}_\theta[|x_n^*P_{B_n}\widehat{P}_mTy_\theta|^p].
    \end{equation*}
    By Jensen's inequality,
    \begin{equation*}
        \mathbb{E}_\theta[|x_n^*P_{B_n}\widehat{P}_mTy_\theta|^p]\geq \left|\mathbb{E}_\theta[\theta_nx_n^*P_{B_n}\widehat{P}_mTy_\theta]\right|^p.
    \end{equation*}
    Since $\mathbb{E}_\theta[\theta_n\theta_r]=0$ whenever $n\neq r$ and $\mathbb{E}_\theta[\theta_n^2]=1$, we get
    \begin{equation*}
        \left|\mathbb{E}_\theta[\theta_nx_n^*P_{B_n}\widehat{P}_mTy_\theta]\right|=|x_n^*P_{B_n}\widehat{P}_mTJ_ny_n|>\varepsilon/2.
    \end{equation*}
    Hence
    \begin{equation*}
        \norm{T}^p>N(\varepsilon/2)^p,
    \end{equation*}
    contradicting the hypothesis.
\end{proof}

Using this, we can give the desired upper triangular reduction.

\begin{proposition}[Reduction to upper-triangular operators from self-similarity]\label{prop:upper-triangular-reduction-from-reprod-ellinfty}
    Let $1 \leq p < \infty$, $E = c_0$ or $E = \ell_\infty$, $X$ be a Banach space, and $T\colon E(\ell_p(X)) \to E(\ell_p(X))$ be an operator. Then, for every $\varepsilon > 0$, there is an infinite set $M\subseteq\N$ and, for each $m\in M$, an infinite set $A_m\subseteq\N$ such that
    \begin{equation*}
        \sum_{\substack{n<m\\ n\in M}}\norm{P_{A_m}\widehat{P}_mT\widehat{J}_n}<\varepsilon.
    \end{equation*}
\end{proposition}

\begin{proof}
    For each $j\in\N$, choose $N_j\in\N$ such that
    \begin{equation*}
        N_j^{1/p}\varepsilon/2^{j+1}>2\norm{T}.
    \end{equation*}
    Using \Cref{lmm:forward-reduction-from-reprod-ellinfty}, we recursively construct an increasing sequence $(m_j)_{j\in\N}$, a decreasing sequence $(M_j)_{j\in\N_0}$ of infinite subsets of $\N$, with $M_0=\N$, and, for each $j\in\N$ and each $m\in M_j$, an infinite set $B_m^j\subseteq\N$. The construction will ensure that $m_j\in M_{j-1}$, $M_j\subseteq M_{j-1}\cap\{m\in\N:m>m_j\}$, and, whenever $1\leq i\leq j$ and $m\in M_j$, we have
    \begin{equation*}
        \norm{P_{B_m^j}\widehat{P}_mT\widehat{J}_{m_i}}<\varepsilon/2^{i+1}.
    \end{equation*}
    Set $B_m^0=\N$ for every $m\in M_0$.

    Suppose that the construction has been carried out up to stage $j-1$. Choose distinct elements $a_1<\ldots<a_{N_j}$ of $M_{j-1}$, let $F=\{a_1,\ldots,a_{N_j}\}$ and define
    \begin{equation*}
        \widehat{J}_F\colon \ell_\infty^{N_j}(\ell_p(X))\to E(\ell_p(X))
    \end{equation*}
    by
    \begin{equation*}
        \widehat{J}_F(x_1,\ldots,x_{N_j})=\sum_{s=1}^{N_j}\widehat{J}_{a_s}x_s.
    \end{equation*}
    Put
    \begin{equation*}
        L_j=M_{j-1}\cap\{m\in\N:m>a_{N_j}\}.
    \end{equation*}
    After identifying each $\ell_p(B_m^{j-1},X)$ with $\ell_p(X)$, define $R_j\colon \ell_\infty^{N_j}(\ell_p(X))\to E(L_j,\ell_p(X))$ coordinate-wise by
    \begin{equation*}
        \widehat{P}_mR_jx=P_{B_m^{j-1}}\widehat{P}_mT\widehat{J}_Fx \qquad \text{for } m\in L_j.
    \end{equation*}
    Then $\norm{R_j}\leq\norm{T}$. Applying \Cref{lmm:forward-reduction-from-reprod-ellinfty} to $R_j$, after identifying $E(L_j,\ell_p(X))$ with $E(\ell_p(X))$ and translating the outer coordinates and inner coordinate sets back through the previous identifications, we obtain $1\leq s_j\leq N_j$, an infinite set $M_j\subseteq L_j$, and, for each $m\in M_j$, an infinite set $B_m^j\subseteq B_m^{j-1}$ such that
    \begin{equation*}
        \norm{P_{B_m^j}\widehat{P}_mT\widehat{J}_{a_{s_j}}}<\varepsilon/2^{j+1}.
    \end{equation*}
    Set $m_j=a_{s_j}$. Since $B_m^j\subseteq B_m^{j-1}$, all previous estimates are preserved. This completes the recursive construction.

    Let $M=\{m_j:j\in\N\}$ and for each $j\in\N$, define $A_{m_j}=B_{m_j}^{j-1}$.
    
    If $i<j$, then $m_j\in M_i$, and since $A_{m_j}\subseteq B_{m_j}^i$, the estimate constructed at stage $i$ gives
    \begin{equation*}
        \norm{P_{A_{m_j}}\widehat{P}_{m_j}T\widehat{J}_{m_i}}<\varepsilon/2^{i+1}.
    \end{equation*}
    Therefore, for every $j\in\N$,
    \begin{equation*}
        \sum_{\substack{n<m_j\\ n\in M}}\norm{P_{A_{m_j}}\widehat{P}_{m_j}T\widehat{J}_n}=\sum_{i<j}\norm{P_{A_{m_j}}\widehat{P}_{m_j}T\widehat{J}_{m_i}}<\sum_{i<j}\varepsilon/2^{i+1}<\varepsilon.
    \end{equation*}
    This proves the result.
\end{proof}
\section{Lower triangular reduction}\label{sec:lower-triangular}

In this section we prove the corresponding lower-triangular reductions. The general mechanism is again based on a one-step reduction: a backward reduction lemma allows us, after passing to suitable infinite subsets, to choose a coordinate whose image has arbitrarily small projection onto prescribed earlier coordinates. A standard recursive selection then turns these estimates into a lower-triangular form.

We begin with the case of $\ell_1(X)$. The finite cotype of the dual space case is obtained from \Cref{lmm:dual-cotype-backward-reduction-ell1} and \Cref{prop:lower-triangular-ell1-dual-cotype}, while the self-similar case is obtained from \Cref{lmm:backward-reduction-fromreprod-ell1} and \Cref{prop:lower-triangular-from-self-similarity-ell1}. We then treat the $c_0(X)$ and $\ell_\infty(X)$ cases: finite cotype is handled in \Cref{lmm:cotype-backward-reduction} and \Cref{prop:prop-lower-triangular-ellinfty-cotype}, and self-similarity in \Cref{lmm:backward-reduction-fromreprod-ellinfty} and \Cref{prop:lower-triangular-from-self-similarity-ellinfty}.

\subsection{The \texorpdfstring{$\ell_1$}{ell_1} case}

\subsubsection{Finite cotype of the dual space}

We begin by considering the operators on $\ell_1(X)$, again under the assumption that $X^*$ has finite cotype. The following lemma is the adjoint counterpart of the finite cotype estimate used in the proof of \Cref{lmm:cotype-forward-reduction}.

\begin{lemma}[Backward reduction from dual cotype]\label{lmm:dual-cotype-backward-reduction-ell1}
    Let $X$ and $Y$ be Banach spaces, suppose that $X^*$ has cotype $q < \infty$, and let $C_q(X^*)$ denote its cotype $q$ constant. Let $\varepsilon>0$ and $T\colon X\to \ell_1^N(Y)$ be an operator. Suppose that
    \begin{equation*}
        N^{1/q}\varepsilon>2C_q(X^*)\norm{T}.
    \end{equation*}
    Then there exists $1\leq n\leq N$ such that
    \begin{equation*}
        \norm{P_nT}<\varepsilon.
    \end{equation*}
\end{lemma}

\begin{proof}
    We proceed by contradiction and suppose that $\norm{P_nT}\geq\varepsilon$ for every $1\leq n\leq N$. For each $n$, choose $y_n^*\in B_{Y^*}$ such that
    \begin{equation*}
        \norm{(P_nT)^*y_n^*}>\varepsilon/2.
    \end{equation*}
    Identifying $(\ell_1^N(Y))^*$ with $\ell_\infty^N(Y^*)$, we have $(P_nT)^*y_n^*=T^*J_ny_n^*$, where $J_ny_n^*$ denotes the element of $\ell_\infty^N(Y^*)$ supported in the $n$-th coordinate.

    For each choice of signs $\theta=(\theta_1,\ldots,\theta_N)\in\{-1,1\}^N$, put
    \begin{equation*}
        y_\theta^*=\sum_{n=1}^N\theta_nJ_ny_n^*.
    \end{equation*}
    Then $\norm{y_\theta^*}_{\ell_\infty^N(Y^*)}\leq 1$, and therefore
    \begin{equation*}
        \norm{T}^q=\norm{T^*}^q\geq \mathbb{E}_\theta[\norm{T^*y_\theta^*}^q].
    \end{equation*}
    By the cotype $q$ estimate in $X^*$,
    \begin{equation*}
        \mathbb{E}_\theta[\norm{T^*y_\theta^*}^q]=\mathbb{E}_\theta\left[\left\|\sum_{n=1}^N\theta_nT^*J_ny_n^*\right\|^q\right]\geq C_q(X^*)^{-q}\sum_{n=1}^N\norm{T^*J_ny_n^*}^q.
    \end{equation*}
    Hence
    \begin{equation*}
        \norm{T}^q>C_q(X^*)^{-q}N(\varepsilon/2)^q,
    \end{equation*}
    and so
    \begin{equation*}
        N^{1/q}\varepsilon<2C_q(X^*)\norm{T},
    \end{equation*}
    contradicting the hypothesis.
\end{proof}

\begin{remark}
    The reason this reduction is effective for $\ell_1$-sums is that restrictions to coordinate sets are controlled by the individual coordinate restrictions. Indeed, if $S\colon \ell_1(X)\to Y$ is bounded and $M\subseteq\mathbb{N}$, then
    \begin{equation*}
        \norm{S J_M}=\sup_{m\in M}\norm{S J_m}.
    \end{equation*}
    Hence, coordinatewise smallness on $M$ immediately yields smallness of the restricted operator $S J_M$. One may prove analogous coordinatewise backward reductions for other $\ell_p$-sums, but such estimates no longer control $SJ_M$ itself, since the domain norm is not the $\ell_1$-norm.
\end{remark}

Using the previous result, we can give the lower-triangular reduction in the case of Banach spaces $X$ so that $X^*$ has finite cotype.

\begin{proposition}[Reduction to lower-triangular operators from dual cotype]\label{prop:lower-triangular-ell1-dual-cotype}
    Let $X$ be a Banach space such that $X^*$ has finite cotype. Then, for every operator $T\colon \ell_1(X)\to \ell_1(X)$ and every $\varepsilon>0$, there exists an infinite subset $M\subseteq\N$ such that
    \begin{equation*}
        \norm{P_MTJ_M-L}<\varepsilon,
    \end{equation*}
    where $L\colon \ell_1(M,X)\to\ell_1(M,X)$ is the lower triangular part of $P_MTJ_M$.
\end{proposition}

\begin{proof}
    Let $q<\infty$ be such that $X^*$ has cotype $q$, and let $C_q(X^*)$ denote its cotype $q$ constant. For each $k\in\N$, choose $N_k\in\N$ such that
    \begin{equation*}
        N_k^{1/q}\varepsilon/2^{k+1}>2C_q(X^*)\norm{T}.
    \end{equation*}

    Using \Cref{lmm:dual-cotype-backward-reduction-ell1}, we recursively construct an increasing sequence $(n_k)_{k\in\N}$ and infinite sets
    \begin{equation*}
        \N=A_0\supseteq A_1\supseteq A_2\supseteq\ldots
    \end{equation*}
    such that $n_k\in A_{k-1}$, $A_k\subseteq A_{k-1}\cap\{j\in\N:j>n_k\}$, and, whenever $1\leq i\leq k$ and $j\in A_k$, we have
    \begin{equation*}
        \norm{P_{n_i}TJ_j}<\varepsilon/2^{i+1}.
    \end{equation*}

    Suppose the construction has been carried out up to stage $k-1$. Choose distinct elements $a_1<\ldots<a_{N_k}$ of $A_{k-1}$, and put
    \begin{equation*}
        L_k=\{j\in A_{k-1}:j>a_{N_k}\}.
    \end{equation*}
    For each $j\in L_k$, define $T_j\colon X\to \ell_1^{N_k}(X)$ by
    \begin{equation*}
        T_jx=(P_{a_s}TJ_jx)_{s=1}^{N_k}.
    \end{equation*}
    Then $\norm{T_j}\leq\norm{T}$. Hence, by \Cref{lmm:dual-cotype-backward-reduction-ell1}, there exists $1\leq s(j)\leq N_k$ such that
    \begin{equation*}
        \norm{P_{a_{s(j)}}TJ_j}<\varepsilon/2^{k+1}.
    \end{equation*}
    By the pigeonhole principle, there are $1\leq s_k\leq N_k$ and an infinite set $A_k\subseteq L_k$ such that $s(j)=s_k$ for every $j\in A_k$. Define $n_k=a_{s_k}$. Since $A_k\subseteq A_{k-1}$, all previous estimates are preserved, and the new estimate follows from the choice of $A_k$. This completes the recursive construction.

    Set $M=\{n_k:k\in\N\}$ and let $S=P_MTJ_M$. Write $S=(S_{i,j})_{i,j\in\N}$ with respect to the enumeration $(n_k)_{k\in\N}$ of $M$, so that
    \begin{equation*}
        S_{i,j}=P_{n_i}TJ_{n_j}.
    \end{equation*}
    If $i<j$, then $n_j\in A_i$, and therefore
    \begin{equation*}
        \norm{S_{i,j}}<\varepsilon/2^{i+1}.
    \end{equation*}

    Define $H\colon \ell_1(M,X)\to\ell_1(M,X)$ to be the strictly upper triangular part of $S$. Then, for every $x\in\ell_1(M,X)$,
    \begin{equation*}
        \norm{Hx}\leq \sum_{j=2}^\infty\left(\sum_{i=1}^{j-1}\frac{\varepsilon}{2^{i+1}}\right)\norm{x_{n_j}}\leq \frac{\varepsilon}{2}\norm{x}.
    \end{equation*}
    Hence $H$ is bounded and $\norm{H}<\varepsilon$. Now set $L=S-H$, so that $L$ is the lower triangular part of $S=P_MTJ_M$, and
    \begin{equation*}
        \norm{P_MTJ_M-L}=\norm{H}<\varepsilon.
    \end{equation*}
\end{proof}

\subsubsection{Self-similarity of the space}

We obtain an analogous lower triangular reduction in the cases that the space has $\ell_p$ or $c_0$-self-similarity. First, we shall need the following preliminary result.

\begin{lemma}[Backward reduction from self-similarity]\label{lmm:backward-reduction-fromreprod-ell1}
    Let $1 < p < \infty$, $X$ and $Y$ be Banach spaces, $\varepsilon>0$, and $T\colon \ell_p(X) \to \ell_1^N(Y)$ be an operator. Suppose that
    \begin{equation*}
        N^{(p-1)/p}\varepsilon > 2 \norm{T}.
    \end{equation*}
    Then there exists $1\leq n\leq N$ and $M \subseteq \N$ infinite such that
    \begin{equation*}
        \norm{P_n T J_M} < \varepsilon.
    \end{equation*}
    The same result holds for operators from $c_0(X)$ and $\ell_\infty(X)$ provided $N \varepsilon > 2 \norm{T}$.
\end{lemma}

\begin{proof}
    As before, we proceed by contradiction and assume this is not the case. Let $M_1, \dots, M_N$ be a partition of $\N$ into infinite sets. By assumption $\norm{P_n T J_{M_n}} \geq \varepsilon$ for $1 \leq n \leq N$, so we can choose $x_n \in B_{\ell_p(M_n, X)}$ and $y_n^* \in B_{Y^*}$ such that
    \begin{equation*}
        |y_n^* P_n T J_{M_n} x_n| > \varepsilon/2.
    \end{equation*}
    For each choice of signs $\theta \in \{-1, 1\}^N$ let
    \begin{equation*}
        x_\theta = \sum_{n = 1}^N \theta_n J_{M_n} x_n.
    \end{equation*}
    Since $J_{M_1} x_1, \dots J_{M_N} x_N$ are disjointly supported vectors of norm at most one, we have that $\norm{x_\theta} \leq N^{1/p}$, so that
    \begin{equation*}
        \norm{T x_\theta} \leq \norm{T} N^{1/p} 
    \end{equation*}
    for any choice of signs $\theta$. On the other hand, averaging, we have
    \begin{equation*}
        \mathbb{E}_\theta [\norm{T x_\theta}] = \sum_{n=1}^N \mathbb{E}_\theta [\norm{P_n T x_\theta}] \geq \sum_{n=1}^N |\mathbb{E}_\theta [y^*_n \theta_n P_n T x_\theta]|.
    \end{equation*}
    Since $\mathbb{E}_\theta[\theta_n\theta_m]=0$ whenever $n\neq m$ and $\mathbb{E}_\theta[\theta_n^2]=1$, we get
    \begin{equation*}
        \left|\mathbb{E}_\theta [\theta_n y_n^* P_n T x_\theta] \right|=|y_n^* P_n T J_{M_n} x_n| > \varepsilon/2.
    \end{equation*}
    It follows that
    \begin{equation*}
        N \varepsilon/2 < \mathbb{E}_\theta [\norm{T x_\theta}] \leq \norm{T} N^{1/p},
    \end{equation*}
    in other words $N^{(p-1)/p} \varepsilon < 2 \norm{T}$, which gives the desired contradiction.

    The proof for $c_0$ and $\ell_\infty$ sums is identical, but now $\norm{x_\theta} \leq 1$, yielding
    \begin{equation*}
        N\varepsilon/2 < \mathbb{E}_\theta [\norm{T x_\theta}] \leq \norm{T},
    \end{equation*}
    which contradicts $N\varepsilon > 2\norm{T}$.
\end{proof}

We can now give a version of lower triangular reduction in the case of $\ell_p$ and $c_0$ self-similarity.

\begin{proposition}[Reduction to lower-triangular operators from self-similarity]\label{prop:lower-triangular-from-self-similarity-ell1}
    Let $E = \ell_p$ for $1 < p \leq \infty$ or $E = c_0$, $X$ be a Banach space and $T\colon \ell_1(E(X))\to \ell_1(E(X))$ be an operator. Then, for every $\varepsilon > 0$, there exist an increasing sequence $\mathbf{n}=(n_j)_{j\in\N}$ and infinite sets $(A_j)_{j \in \N}$ so that for all $i \in \N$ we have that whenever $j > i$ then
    \begin{equation*}
        \norm{\widehat{P}_{n_i} T \widehat{J}_{n_j} J_{A_j}} < \varepsilon/2^{i+1}.
    \end{equation*}
\end{proposition}
\begin{proof}
    We first do the case $E=\ell_p$ for $1<p<\infty$. For each $k\in\N$, choose $N_k\in\N$ such that
    \begin{equation*}
        N_k^{(p-1)/p}\varepsilon/2^{k+1}>2\norm{T}.
    \end{equation*}

    Using \Cref{lmm:backward-reduction-fromreprod-ell1}, we recursively construct an increasing sequence $(n_k)_{k\in\N}$, a decreasing sequence $(M_k)_{k\in\N_0}$ of infinite subsets of $\N$ with $M_0=\N$, and, for each $k\in\N_0$ and each $m\in M_k$, an infinite set $B_m^k\subseteq\N$. The construction will ensure that $n_k\in M_{k-1}$, $M_k\subseteq M_{k-1}\cap\{m\in\N:m>n_k\}$, and, whenever $1\leq i\leq k$ and $m\in M_k$, we have
    \begin{equation*}
        \norm{\widehat{P}_{n_i}T\widehat{J}_mJ_{B_m^k}}<\varepsilon/2^{i+1}.
    \end{equation*}
    Set $B_m^0=\N$ for every $m\in M_0=\N$.

    Suppose the construction has been carried out up to stage $k-1$. Choose distinct elements $a_1<\ldots<a_{N_k}$ of $M_{k-1}$, and let $B=\{m\in M_{k-1}:m>a_{N_k}\}$. For each $m\in B$, define
    \begin{equation*}
        T_m\colon \ell_p(B_m^{k-1},X)\to \ell_1^{N_k}(\ell_p(X)), \qquad T_mx=(\widehat{P}_{a_s}T\widehat{J}_mJ_{B_m^{k-1}}x)_{s=1}^{N_k}.
    \end{equation*}
    Then $\norm{T_m}\leq\norm{T}$. Hence, by \Cref{lmm:backward-reduction-fromreprod-ell1}, there exist $1\leq s(m)\leq N_k$ and an infinite set $L_m\subseteq B_m^{k-1}$ such that
    \begin{equation*}
        \norm{\widehat{P}_{a_{s(m)}}T\widehat{J}_mJ_{L_m}}<\varepsilon/2^{k+1}.
    \end{equation*}
    By the pigeonhole principle, there exist $1\leq s_k\leq N_k$ and an infinite set $M_k\subseteq B$ such that $s(m)=s_k$ for every $m\in M_k$. Define $ n_k=a_{s_k}$. For $m\in M_k$, set $B_m^k=L_m$. Since $B_m^k\subseteq B_m^{k-1}$, all estimates obtained at previous stages are preserved, and the estimate for the new row follows from the choice of $M_k$. This completes the recursive construction.

    For each $j\in\N$, define
    \begin{equation*}
        A_j=B_{n_j}^{j-1}.
    \end{equation*}
    If $j>i$, then $n_j\in M_i$, and since $A_j\subseteq B_{n_j}^i$, the estimate constructed at stage $i$ gives
    \begin{equation*}
        \norm{\widehat{P}_{n_i}T\widehat{J}_{n_j}J_{A_j}}<\varepsilon/2^{i+1}.
    \end{equation*}

    The cases $E=\ell_\infty$ and $E=c_0$ are identical, except that at the beginning we choose $N_k\in\N$ such that
    \begin{equation*}
        N_k\varepsilon/2^{k+1}>2\norm{T},
    \end{equation*}
    and we apply the $c_0$ or $\ell_\infty$ part of \Cref{lmm:backward-reduction-fromreprod-ell1}.
\end{proof}

\subsection{The \texorpdfstring{$\ell_\infty$ and $c_0$}{ell-infinity and c0} cases}

\subsubsection{Finite cotype}

We now turn to the $c_0(X)$ and $\ell_\infty(X)$ cases under the assumption that $X$ has finite cotype. The backward reduction is simpler than in the $\ell_1$ case: finite cotype allows us to make the restriction of an operator to a suitable infinite set of coordinates arbitrarily small. This is the content of the next lemma.

\begin{lemma}[Backward reduction from cotype]\label{lmm:cotype-backward-reduction}
    Let $X$ have cotype $q<\infty$, $E$ be either $c_0$ or $\ell_\infty$, and $T\colon E(X)\to X$ be an operator. Then, for every $\varepsilon > 0$ there exists an infinite subset $M\subseteq \mathbb{N}$ such that $\norm{T J_M}<\varepsilon$.
\end{lemma}

\begin{proof}
    We proceed by contradiction and suppose this is not the case. Then there exists $\varepsilon>0$ so that
    \begin{equation*}
        \norm{T J_M}\geq\varepsilon
    \end{equation*}
    for every infinite subset $M\subseteq \mathbb{N}$. Choose $N\in\N$ such that
    \begin{equation*}
        N^{1/q}\varepsilon>2C_q(X)\norm{T}.
    \end{equation*}
    Let $M_1,\ldots,M_N$ be pairwise disjoint infinite subsets of $\mathbb{N}$. By the assumption, for each $1\leq n\leq N$ we may choose $x_n\in B_{E(M_n,X)}$ such that
    \begin{equation*}
        \norm{TJ_{M_n}x_n}>\varepsilon/2.
    \end{equation*}
    For each choice of signs $\theta=(\theta_1,\ldots,\theta_N)\in\{-1,1\}^N$, put
    \begin{equation*}
        x_\theta=\sum_{n=1}^N\theta_nJ_{M_n}x_n.
    \end{equation*}
    Since the sets $M_1,\ldots,M_N$ are pairwise disjoint, we have $\norm{x_\theta}_{E(X)}\leq 1$. Therefore
    \begin{equation*}
        \norm{T}^q\geq \mathbb{E}_\theta [\norm{Tx_\theta}^q]=\mathbb{E}_\theta\left[\left\|\sum_{n=1}^N\theta_nTJ_{M_n}x_n\right\|^q\right].
    \end{equation*}
    By cotype $q$ of $X$,
    \begin{equation*}
        \mathbb{E}_\theta\left[\left\|\sum_{n=1}^N\theta_nTJ_{M_n}x_n\right\|^q\right]\geq C_q(X)^{-q}\sum_{n=1}^N\norm{TJ_{M_n}x_n}^q>C_q(X)^{-q}N(\varepsilon/2)^q.
    \end{equation*}
    Hence
    \begin{equation*}
        N^{1/q}\varepsilon<2C_q(X)\norm{T},
    \end{equation*}
    contradicting the choice of $N$.
\end{proof}

The lower-triangular reduction will now follow from the same recursive selection argument.

\begin{proposition}[Reduction to lower-triangular operators from cotype]\label{prop:prop-lower-triangular-ellinfty-cotype}
    Let $X$ be a Banach space with finite cotype $q$, $E = c_0$ or $E = \ell_\infty$ and $T: E(X) \to E(X)$ be an operator. Then for every $\varepsilon > 0$, there exists an infinite subset $M \subseteq \N$ such that
    \begin{equation*}
        \norm{P_M T J_M - L} < \varepsilon,
    \end{equation*}
    where $L \colon E(M, X) \to E(M, X)$ is the lower triangular part of $P_M T J_M$.
\end{proposition}
\begin{proof}
    Using \Cref{lmm:cotype-backward-reduction}, we recursively construct an increasing sequence $(n_k)_{k\in\N}$ and infinite sets $(A_k)_{k\in\N_0}$, with $A_0=\N$, such that $n_k\in A_{k-1}$, $A_k\subseteq A_{k-1}\cap\{n\in\N:n>n_k\}$, and
    \begin{equation*}
        \norm{P_{n_k}TJ_{A_k}}<\varepsilon/2^{k+1}.
    \end{equation*}

    We now do the recursive construction. Assume that $A_{k-1}$ has been constructed. Choose $n_k\in A_{k-1}$, and put
    \begin{equation*}
        B_k=A_{k-1}\cap\{n\in\N:n>n_k\}.
    \end{equation*} 
    Applying \Cref{lmm:cotype-backward-reduction} to the operator
    \begin{equation*}
        P_{n_k}TJ_{B_k}\colon E(B_k,X)\to X,
    \end{equation*}
    after identifying $E(B_k,X)$ with $E(X)$, we obtain an infinite set $A_k\subseteq B_k$ such that
    \begin{equation*}
        \norm{P_{n_k}TJ_{A_k}}<\varepsilon/2^{k+1}.
    \end{equation*}
    This completes the recursive construction.

    Set $M=\{n_k:k\in\N\}$ and, for each $k\in\N$, put $M_k=\{n_j:j>k\}$. Since $M_k\subseteq A_k$, we have
    \begin{equation*}
        \norm{P_{n_k}TJ_{M_k}}<\varepsilon/2^{k+1}
    \end{equation*}
    for every $k\in\N$.

    Let $S=P_MTJ_M$. For $x\in E(M,X)$, define $Hx$ coordinatewise by
    \begin{equation*}
        P_{n_k}Hx=P_{n_k}TJ_{M_k}P_{M_k}x
    \end{equation*}
    for every $k\in\N$. Then, for every $x\in E(M,X)$ and every $k\in\N$, we have
    \begin{equation*}
        \norm{P_{n_k}Hx}\leq \frac{\varepsilon}{2^{k+1}}\norm{x}.
    \end{equation*}
    Hence $H$ is a bounded operator on $\ell_\infty(M,X)$ with $\norm{H}<\varepsilon$. In the case $E=c_0$, the same estimate also shows that $Hx\in c_0(M,X)$ whenever $x\in c_0(M,X)$. Define $L=S-H$. Then
    \begin{equation*}
        \norm{P_MTJ_M-L}=\norm{H}<\varepsilon.
    \end{equation*}
    Finally, for every $x\in E(M,X)$ and every $k\in\N$, we have
    \begin{equation*}
        P_{n_k}Lx=P_{n_k}TJ_{\{n_1,\ldots,n_k\}}P_{\{n_1,\ldots,n_k\}}x=\sum_{j=1}^kP_{n_k}TJ_{n_j}x_{n_j}.
    \end{equation*}
    Thus $L$ has the lower triangular matrix representation of $P_MTJ_M$ and acts according to this matrix. Hence, also in the case $E=\ell_\infty$, $L$ is the lower triangular part of $P_MTJ_M$.
\end{proof}

\subsubsection{Self-similarity of the space}

It remains to treat the self-similar case for $c_0$ and $\ell_\infty$ sums. Here the backward reduction follows from the incompatibility between the sup norm in the domain and the $\ell_p$-structure in the target; the lower-triangular reduction then follows by the usual recursive selection.

\begin{lemma}[Backward reduction from self-similarity]\label{lmm:backward-reduction-fromreprod-ellinfty}
    Let $X$ and $Y$ be Banach spaces, $E = c_0$ or $E = \ell_\infty$ and $T\colon E(Y) \to \ell_p(X)$ be an operator where $1 \leq p < \infty$. Then for any $\varepsilon > 0$, there exist infinite sets $A, M \subseteq \N$ such that $\norm{P_A T J_M} < \varepsilon$.
\end{lemma}
\begin{proof}
    We proceed by contradiction and suppose that the conclusion fails. Then there is $\varepsilon>0$ such that
    \begin{equation*}
        \norm{P_ATJ_M}\geq\varepsilon
    \end{equation*}
    for every pair of infinite sets $A,M\subseteq\N$. Choose $N\in\N$ such that
    \begin{equation*}
        N^{1/p}\varepsilon>2\norm{T}.
    \end{equation*}
    Let $A_1,\ldots,A_N$ and $M_1,\ldots,M_N$ be partitions of $\N$ into infinite sets. By assumption, for each $1\leq n\leq N$ we have
    \begin{equation*}
        \norm{P_{A_n}TJ_{M_n}}\geq\varepsilon.
    \end{equation*}
    Hence we may choose $y_n\in B_{E(M_n,Y)}$ and $x_n^*\in B_{\ell_p(A_n,X)^*}$ such that
    \begin{equation*}
        |x_n^*P_{A_n}TJ_{M_n}y_n|>\varepsilon/2.
    \end{equation*}

    For each choice of signs $\theta=(\theta_1,\ldots,\theta_N)\in\{-1,1\}^N$, put
    \begin{equation*}
        y_\theta=\sum_{n=1}^N\theta_nJ_{M_n}y_n.
    \end{equation*}
    Since $E=c_0$ or $E=\ell_\infty$ and the vectors $J_{M_n}y_n$ have pairwise disjoint supports, we have $\norm{y_\theta}\leq 1$. Therefore
    \begin{equation*}
        \norm{T}^p\geq \mathbb{E}_\theta[\norm{Ty_\theta}^p].
    \end{equation*}
    Since the sets $A_1,\ldots,A_N$ are pairwise disjoint, we obtain
    \begin{equation*}
        \mathbb{E}_\theta[\norm{Ty_\theta}^p]\geq \sum_{n=1}^N\mathbb{E}_\theta[\norm{P_{A_n}Ty_\theta}^p]\geq \sum_{n=1}^N\mathbb{E}_\theta[|x_n^*P_{A_n}Ty_\theta|^p].
    \end{equation*}
    By Jensen's inequality,
    \begin{equation*}
        \mathbb{E}_\theta[|x_n^*P_{A_n}Ty_\theta|^p]\geq \left|\mathbb{E}_\theta[\theta_nx_n^*P_{A_n}Ty_\theta]\right|^p.
    \end{equation*}
    Since $\mathbb{E}_\theta[\theta_n\theta_m]=0$ whenever $n\neq m$ and $\mathbb{E}_\theta[\theta_n^2]=1$, we have
    \begin{equation*}
        \left|\mathbb{E}_\theta[\theta_nx_n^*P_{A_n}Ty_\theta]\right|=|x_n^*P_{A_n}TJ_{M_n}y_n|>\varepsilon/2.
    \end{equation*}
    It follows that
    \begin{equation*}
        \norm{T}^p\geq \mathbb{E}_\theta[\norm{Ty_\theta}^p]>N(\varepsilon/2)^p,
    \end{equation*}
    contradicting the choice of $N$. This contradiction proves the result.
\end{proof}

Lastly, we use the previous lemma to give the desired version of lower-triangular reduction.

\begin{proposition}[Reduction to lower-triangular operators from self-similarity]\label{prop:lower-triangular-from-self-similarity-ellinfty}
    Let $1\leq p<\infty$, $E=c_0$ or $E=\ell_\infty$, $X$ be a Banach space, and $T\colon E(\ell_p(X)) \to E(\ell_p(X))$ be an operator. Then, for every $\varepsilon>0$, there exists an infinite set $M\subseteq\N$ and, for each $m\in M$, an infinite set $A_m\subseteq\N$ such that, for every $m\in M$,
    \begin{equation*}
        \norm{P_{A_m}\widehat{P}_mT\widehat{J}_{M_m}}<\varepsilon,
    \end{equation*}
    where $M_m=\{n\in M:n>m\}$.
\end{proposition}
\begin{proof}
    Using \Cref{lmm:backward-reduction-fromreprod-ellinfty}, we recursively construct an increasing sequence $(m_k)_{k\in\N}$, a decreasing sequence $(B_k)_{k\in\N_0}$ of infinite subsets of $\N$, with $B_0=\N$, and infinite sets $(A_{m_k})_{k\in\N}$ such that $m_k\in B_{k-1}$, $B_k\subseteq B_{k-1}\cap\{m\in\N:m>m_k\}$, and
    \begin{equation*}
        \norm{P_{A_{m_k}}\widehat{P}_{m_k}T\widehat{J}_{B_k}}<\varepsilon.
    \end{equation*}

    Assume that $B_{k-1}$ has been constructed. Choose $m_k\in B_{k-1}$, and put
    \begin{equation*}
        L_k=B_{k-1}\cap\{m\in\N:m>m_k\}.
    \end{equation*}
    Then $L_k$ is infinite. Applying \Cref{lmm:backward-reduction-fromreprod-ellinfty} to the operator
    \begin{equation*}
        \widehat{P}_{m_k}T\widehat{J}_{L_k}\colon E(L_k,\ell_p(X))\to \ell_p(X),
    \end{equation*}
    after identifying $E(L_k,\ell_p(X))$ with $E(\ell_p(X))$, we obtain an infinite set $A_{m_k}\subseteq\N$ and, after translating back through the identification, an infinite set $B_k\subseteq L_k$ such that
    \begin{equation*}
        \norm{P_{A_{m_k}}\widehat{P}_{m_k}T\widehat{J}_{B_k}}<\varepsilon.
    \end{equation*}
    This completes the recursive construction.

    Set $M=\{m_k:k\in\N\}$. Let $m\in M$, say $m=m_k$. Since
    \begin{equation*}
        M_m=\{n\in M:n>m\}=\{m_j:j>k\}\subseteq B_k,
    \end{equation*}
    we have
    \begin{equation*}
        \norm{P_{A_m}\widehat{P}_mT\widehat{J}_{M_m}}\leq \norm{P_{A_{m_k}}\widehat{P}_{m_k}T\widehat{J}_{B_k}}<\varepsilon.
    \end{equation*}
    This proves the result.
\end{proof}
\bigskip
\section{Diagonal reductions and proof of the main theorems}\label{sec:diagonal-and-proofs}

In this section we combine the upper- and lower-triangular reductions from the previous sections. Applying both reduction allow us to pass, up to an arbitrarily small perturbation, to the diagonal part. These diagonal reductions are the final structural step before the factorisation argument proving the main theorems.

\subsection{Diagonal reductions}

We prove the diagonal reductions. The proofs are obtained by applying the lower-triangular reductions of \Cref{sec:lower-triangular} and then the upper-triangular reductions of \Cref{sec:upper-triangular}, which gives the desired diagonal form.

\subsubsection{The \texorpdfstring{$\ell_1$}{ell_1} case}

We begin with the $\ell_1$ case, treating first the finite dual-cotype hypothesis and then the self-similar hypothesis.

\begin{proposition}[Diagonal reduction from dual cotype]\label{prop:diagonal-reduction-from-type}
    Let $X$ be a Banach space such that $X^*$ has finite cotype, and let $T\colon \ell_1(X)\to \ell_1(X)$ be an operator. Then, for every $\varepsilon>0$, there exists an infinite subset $M\subseteq\N$ such that
    \begin{equation*}
        \norm{P_MTJ_M-D}<\varepsilon,
    \end{equation*}
    where $D\colon \ell_1(M,X)\to \ell_1(M,X)$ is the diagonal part of $P_MTJ_M$.
\end{proposition}

\begin{proof}
    Apply \Cref{prop:lower-triangular-ell1-dual-cotype} to $T$ with $\varepsilon/2$ in place of $\varepsilon$. We obtain an infinite subset $M_0\subseteq\N$ such that
    \begin{equation*}
        \norm{P_{M_0}TJ_{M_0}-L}<\varepsilon/2,
    \end{equation*}
    where $L\colon \ell_1(M_0,X)\to \ell_1(M_0,X)$ is the lower triangular part of $P_{M_0}TJ_{M_0}$.

    Identifying $\ell_1(M_0,X)$ with $\ell_1(X)$, apply \Cref{prop:upper-triangular-ell1-dual-cotype} to $L$ with $\varepsilon/2$ in place of $\varepsilon$. Thus, after passing to an infinite subset $M\subseteq M_0$, we have
    \begin{equation*}
        \norm{P_MLJ_M-D}<\varepsilon/2,
    \end{equation*}
    where $D$ is the upper triangular part of $P_MLJ_M$.

    Since $L$ is lower triangular, $P_MLJ_M$ is lower triangular. Hence its upper triangular part $D$ is precisely its diagonal part. Moreover, $L$ has the same diagonal as $P_{M_0}TJ_{M_0}$, and therefore $D$ is also the diagonal part of $P_MTJ_M$. Here, in the terms involving $L$, the projection and inclusion are taken relative to $\ell_1(M_0,X)$. Finally, using the two estimates above and the fact that passing to the smaller set $M\subseteq M_0$ is contractive, we obtain
    \begin{equation*}
        \norm{P_MTJ_M-D}\leq \norm{P_{M_0}TJ_{M_0}-L}+\norm{P_MLJ_M-D}<\varepsilon.
    \end{equation*}
    This proves the result.
\end{proof}

We now have the analogous statement in the case of $\ell_p$ or $c_0$ self-similarity.

\begin{proposition}[Diagonal reduction from self-similarity]\label{prop:diagonal-reduction-from-reprod}
    Let $E = \ell_p$ for $1 < p \leq \infty$ or $E = c_0$ and $T \colon \ell_1(E(X)) \to \ell_1(E(X))$ be an operator. Then, for every $\varepsilon>0$, there exist infinite increasing sequences $\mathbf{n} =(n_j)_{j \in \N}$ and $\mathbf{m} =(m_j)_{j \in \N}$ such that
    \begin{equation*}
        \norm{P_{\mathbf{n}, \mathbf{m}} T J_{\mathbf{n}, \mathbf{m}}-D}<\varepsilon,
    \end{equation*}
    where $D\colon \ell_1(X)\to \ell_1(X)$ is the diagonal part of $P_{\mathbf{n}, \mathbf{m}} T J_{\mathbf{n}, \mathbf{m}}$.
\end{proposition}
\begin{proof}
    Applying \Cref{prop:upper-triangular-reduction-from-reprod-ell1} to $T$, we obtain an increasing sequence $\mathbf{a}=(a_j)_{j\in\N}$ and infinite sets $(B_j)_{j\in\N}$ such that, writing $M_j=\{a_i:i>j\}$, we have
    \begin{equation*}
        \norm{\widehat{P}_{M_j}T\widehat{J}_{a_j}J_{B_j}}<\varepsilon/2^{j+2}
    \end{equation*}
    for every $j\in\N$.

    We shall work on the same inner copies on both the domain and range sides. After identifying each $E(B_j,X)$ isometrically with $E(X)$, let
    \begin{equation*}
        V\colon \ell_1(E(X))\to \ell_1(E(X))
    \end{equation*}
    be the contraction which sends the $j$-th outer coordinate into the $a_j$-th outer coordinate and, inside that coordinate, into $E(B_j,X)$. Let
    \begin{equation*}
        U\colon \ell_1(E(X))\to \ell_1(E(X))
    \end{equation*}
    be the contraction which projects onto the outer coordinates $\{a_i:i\in\N\}$ and, in the $a_i$-th outer coordinate, onto $E(B_i,X)$, followed by the natural reindexing. Set $S=UTV$. Then $S\colon \ell_1(E(X))\to\ell_1(E(X))$ is bounded, and, under the above identifications, its matrix entries are
    \begin{equation*}
        S_{i,j}=P_{B_i}\widehat{P}_{a_i}T\widehat{J}_{a_j}J_{B_j}.
    \end{equation*}

    Applying \Cref{prop:lower-triangular-from-self-similarity-ell1} to $S$ with $\varepsilon/2$ in place of $\varepsilon$, we obtain an increasing sequence $(b_j)_{j\in\N}$ and infinite sets which, after translating back through the identifications $E(B_{b_j},X)\simeq E(X)$, we denote by $A_j\subseteq B_{b_j}$, such that whenever $j>i$,
    \begin{equation*}
        \norm{P_{B_{b_i}}\widehat{P}_{a_{b_i}}T\widehat{J}_{a_{b_j}}J_{A_j}}<\varepsilon/2^{i+2}.
    \end{equation*}
    Set $n_j=a_{b_j}$. Since each $A_j$ is infinite, choose $m_j\in A_j$ recursively so that $(m_j)_{j\in\N}$ is increasing.

    Let $D$ be the diagonal part of $P_{\mathbf{n},\mathbf{m}}TJ_{\mathbf{n},\mathbf{m}}$. We estimate the strictly upper and strictly lower triangular parts separately. Let $x=(x_j)_{j\in\N}\in\ell_1(X)$. For the strictly upper triangular part, the second reduction gives
    \begin{equation*}
        \sum_{i=1}^\infty\left\|\sum_{j>i}P_{m_i}\widehat{P}_{n_i}T\widehat{J}_{n_j}J_{m_j}x_j\right\|\leq \sum_{i=1}^\infty\sum_{j>i}\frac{\varepsilon}{2^{i+2}}\norm{x_j}\leq \frac{\varepsilon}{4}\norm{x}.
    \end{equation*}
    For the strictly lower triangular part, note that $\{n_i:i>j\}\subseteq \{a_i:i>b_j\}=M_{b_j}$ and $A_j\subseteq B_{b_j}$. Hence, by the first reduction,
    \begin{equation*}
        \sum_{i>j}\norm{P_{m_i}\widehat{P}_{n_i}T\widehat{J}_{n_j}J_{m_j}x_j}\leq \norm{\widehat{P}_{M_{b_j}}T\widehat{J}_{a_{b_j}}J_{B_{b_j}}}\norm{x_j}<\frac{\varepsilon}{2^{b_j+2}}\norm{x_j}\leq \frac{\varepsilon}{2^{j+2}}\norm{x_j}.
    \end{equation*}
    Therefore
    \begin{equation*}
        \sum_{j=1}^\infty\sum_{i>j}\norm{P_{m_i}\widehat{P}_{n_i}T\widehat{J}_{n_j}J_{m_j}x_j}\leq \frac{\varepsilon}{4}\norm{x}.
    \end{equation*}
    Combining the two estimates gives
    \begin{equation*}
        \norm{(P_{\mathbf{n},\mathbf{m}}TJ_{\mathbf{n},\mathbf{m}}-D)x}\leq \frac{\varepsilon}{2}\norm{x}.
    \end{equation*}
    Hence
    \begin{equation*}
        \norm{P_{\mathbf{n},\mathbf{m}}TJ_{\mathbf{n},\mathbf{m}}-D}<\varepsilon,
    \end{equation*}
    which finishes the proof.
\end{proof}

\subsubsection{The \texorpdfstring{$\ell_\infty$ and $c_0$}{ell-infinity and c0} cases}

We now prove the corresponding diagonal reductions for $c_0(X)$ and $\ell_\infty(X)$. The finite-cotype case is obtained by combining the lower-triangular reduction from \Cref{prop:prop-lower-triangular-ellinfty-cotype} with the upper-triangular reduction from \Cref{prop:upper-triangular-reduction-from-cotype-ellinfty}.

\begin{proposition}[Diagonal reduction from cotype]\label{prop:diagonal-reduction-from-cotype-ellinfty}
    Let $X$ be a Banach space with finite cotype, $E = c_0$ or $E = \ell_\infty$, and $T\colon E(X) \to E(X)$ be an operator. Then, for every $\varepsilon>0$, there exists an infinite subset $M\subseteq\N$ such that
    \begin{equation*}
        \norm{P_MTJ_M-D}<\varepsilon,
    \end{equation*}
    where $D\colon E(M,X)\to E(M,X)$ is the diagonal part of $P_MTJ_M$.
\end{proposition}

\begin{proof}
    Apply \Cref{prop:prop-lower-triangular-ellinfty-cotype} to $T$ with $\varepsilon/2$ in place of $\varepsilon$. We obtain an infinite subset $M_0\subseteq\N$ such that
    \begin{equation*}
        \norm{P_{M_0}TJ_{M_0}-L}<\varepsilon/2,
    \end{equation*}
    where $L\colon E(M_0,X)\to E(M_0,X)$ is the lower triangular part of $P_{M_0}TJ_{M_0}$.

    Identifying $E(M_0,X)$ with $E(X)$, apply \Cref{prop:upper-triangular-reduction-from-cotype-ellinfty} to $L$ with $\varepsilon/2$ in place of $\varepsilon$. Thus, after passing to an infinite subset $M\subseteq M_0$, we have, for every $m\in M$,
    \begin{equation*}
        \sum_{\substack{n<m\\ n\in M}}\norm{P_mLJ_n}<\varepsilon/2.
    \end{equation*}

    Let $D$ be the diagonal part of $P_MLJ_M$. Since $L$ is lower triangular, $P_MLJ_M-D$ is its strictly lower triangular part. Hence, for every $x\in E(M,X)$ and every $m\in M$,
    \begin{equation*}
        \norm{P_m(P_MLJ_M-D)x}\leq \sum_{\substack{n<m\\ n\in M}}\norm{P_mLJ_n}\norm{x}<\frac{\varepsilon}{2}\norm{x}.
    \end{equation*}
    Therefore
    \begin{equation*}
        \norm{P_MLJ_M-D}\leq\varepsilon/2.
    \end{equation*}
    In the case $E=c_0$, the operator $P_MLJ_M-D$ maps $c_0(M,X)$ into $c_0(M,X)$ because it is the difference of two operators on $c_0(M,X)$.

    Here, in the term $P_MLJ_M$, the projection and inclusion are taken relative to $E(M_0, X)$. Finally, using the two estimates above and the fact that passing to the smaller set $M\subseteq M_0$ is contractive, we obtain
    \begin{equation*}
        \norm{P_MTJ_M-D}\leq \norm{P_{M_0}TJ_{M_0}-L}+\norm{P_MLJ_M-D}<\varepsilon.
    \end{equation*}
    This proves the result.
\end{proof}

The self-similar case is obtained in the same way, using the lower- and upper-triangular reductions from \Cref{prop:lower-triangular-from-self-similarity-ellinfty,prop:upper-triangular-reduction-from-reprod-ellinfty}.

\begin{proposition}[Diagonal reduction from self-similarity]\label{prop:diagonal-reduction-from-reprod-ellinfty}
    Let $1\leq p<\infty$, $E=c_0$ or $E=\ell_\infty$, $X$ be a Banach space, and $T\colon E(\ell_p(X))\to E(\ell_p(X))$ be an operator. Then, for every $\varepsilon>0$, there exist increasing sequences $\mathbf{n}=(n_j)_{j\in\N}$ and $\mathbf{m}=(m_j)_{j\in\N}$ such that
    \begin{equation*}
        \norm{P_{\mathbf{n},\mathbf{m}}TJ_{\mathbf{n},\mathbf{m}}-D}<\varepsilon,
    \end{equation*}
    where $D\colon E(X)\to E(X)$ is the diagonal part of $P_{\mathbf{n},\mathbf{m}}TJ_{\mathbf{n},\mathbf{m}}$.
\end{proposition}

\begin{proof}
    Apply \Cref{prop:lower-triangular-from-self-similarity-ellinfty} to $T$ with $\varepsilon/4$ in place of $\varepsilon$. We obtain an infinite set $M_0\subseteq\N$ and, for each $m\in M_0$, an infinite set $B_m\subseteq\N$ such that, writing $M_m^0=\{n\in M_0:n>m\}$, we have
    \begin{equation*}
        \norm{P_{B_m}\widehat{P}_mT\widehat{J}_{M_m^0}}<\varepsilon/4
    \end{equation*}
    for every $m\in M_0$.

    Let $M_0=\{a_j:j\in\N\}$, where $(a_j)_{j\in\N}$ is increasing. We shall work on the same inner copies on both the domain and range sides. After identifying each $\ell_p(B_{a_j},X)$ isometrically with $\ell_p(X)$, let
    \begin{equation*}
        V\colon E(\ell_p(X))\to E(\ell_p(X))
    \end{equation*}
    be the contraction which sends the $j$-th outer coordinate into the $a_j$-th outer coordinate and, inside that coordinate, into $\ell_p(B_{a_j},X)$. Let
    \begin{equation*}
        U\colon E(\ell_p(X))\to E(\ell_p(X))
    \end{equation*}
    be the contraction which projects onto the outer coordinates $\{a_i:i\in\N\}$ and, in the $a_i$-th outer coordinate, onto $\ell_p(B_{a_i},X)$, followed by the natural reindexing. Set $S=UTV$. Then $S\colon E(\ell_p(X))\to E(\ell_p(X))$ is bounded, and, under the above identifications, its matrix entries are
    \begin{equation*}
        \widehat{P}_iS\widehat{J}_j=P_{B_{a_i}}\widehat{P}_{a_i}T\widehat{J}_{a_j}J_{B_{a_j}}.
    \end{equation*}

    Applying \Cref{prop:upper-triangular-reduction-from-reprod-ellinfty} to $S$ with $\varepsilon/4$ in place of $\varepsilon$, we obtain an infinite set $K\subseteq\N$ and, for each $k\in K$, an infinite set $A_k\subseteq\N$ such that, for every $k\in K$,
    \begin{equation*}
        \sum_{\substack{i<k\\ i\in K}}\norm{P_{A_k}\widehat{P}_kS\widehat{J}_i}<\varepsilon/4.
    \end{equation*}

    Let $K=\{k_j:j\in\N\}$, where $(k_j)_{j\in\N}$ is increasing, and set $n_j=a_{k_j}$. Translating each $A_{k_j}$ back through the identification $\ell_p(B_{n_j},X)\simeq\ell_p(X)$, we obtain an infinite set, still denoted by $A_{k_j}$, with $A_{k_j}\subseteq B_{n_j}$ and
    \begin{equation*}
        \sum_{i<j}\norm{P_{A_{k_j}}\widehat{P}_{n_j}T\widehat{J}_{n_i}J_{B_{n_i}}}<\varepsilon/4
    \end{equation*}
    for every $j\in\N$. Choose $m_j\in A_{k_j}$ recursively so that $(m_j)_{j\in\N}$ is increasing.

    Let $D$ be the diagonal part of $P_{\mathbf{n},\mathbf{m}}TJ_{\mathbf{n},\mathbf{m}}$. We estimate the two off-diagonal parts separately. Let $x=(x_j)_{j\in\N}\in E(X)$. For the terms with $i<j$, we have, for every $j\in\N$,
    \begin{equation*}
        \left\|\sum_{i<j}P_{m_j}\widehat{P}_{n_j}T\widehat{J}_{n_i}J_{m_i}x_i\right\|\leq \sum_{i<j}\norm{P_{A_{k_j}}\widehat{P}_{n_j}T\widehat{J}_{n_i}J_{B_{n_i}}}\norm{x_i}<\frac{\varepsilon}{4}\norm{x}.
    \end{equation*}
    For the terms with $i>j$, note that $\{n_i:i>j\}\subseteq M_{n_j}^0$ and $m_j\in B_{n_j}$. Hence, for every $j\in\N$,
    \begin{equation*}
        \left\|P_{m_j}\widehat{P}_{n_j}T\widehat{J}_{\{n_i:i>j\}}(J_{m_i}x_i)_{i>j}\right\|\leq \norm{P_{B_{n_j}}\widehat{P}_{n_j}T\widehat{J}_{M_{n_j}^0}}\norm{x}<\frac{\varepsilon}{4}\norm{x}.
    \end{equation*}
    Combining the two estimates gives
    \begin{equation*}
        \norm{(P_{\mathbf{n},\mathbf{m}}TJ_{\mathbf{n},\mathbf{m}}-D)x}\leq\frac{\varepsilon}{2}\norm{x}.
    \end{equation*}
    Therefore
    \begin{equation*}
        \norm{P_{\mathbf{n},\mathbf{m}}TJ_{\mathbf{n},\mathbf{m}}-D}<\varepsilon.
    \end{equation*}
    This proves the result.
\end{proof}

\subsection{Proof of the main theorems}

Before we move to the proof of our main results, we state the following two elementary observations.

\begin{proposition}\label{prop:diagonal-dichotomy}
    Let $E=\ell_1$, $E=\ell_\infty$, or $E=c_0$, and $X$ be a Banach space with the $C$-PFP. Then, for every diagonal operator $D\colon E(X)\to E(X)$, either $D$ or $I_{E(X)}-D$ factors the identity on $E(X)$ with constant $C$.
\end{proposition}

\begin{proof}
    Write $D=\diag(D_n:n\in\N)$, where $D_n\in\mathscr{B}(X)$. Since $X$ has the $C$-PFP, for each $n\in\N$ either $D_n$ or $I_X-D_n$ factors the identity on $X$ with constant $C$.

    By passing to an infinite subset $M\subseteq\N$, we may suppose that the same alternative holds for every $n\in M$. Suppose first that $D_n$ factors the identity on $X$ with constant $C$ for every $n\in M$. Thus, for each $n\in M$, there are operators $\Phi_n,\Psi_n\in\mathscr{B}(X)$ such that
    \begin{equation*}
        \Phi_nD_n\Psi_n=I_X \qquad\text{and}\qquad \norm{\Phi_n}\norm{\Psi_n}\leq C.
    \end{equation*}
    After rescaling the factors, we may assume $\norm{\Phi_n}\leq C$ and $\norm{\Psi_n}\leq 1$ for every $n\in M$. Define diagonal operators
    \begin{equation*}
        \Phi=\diag(\Phi_n:n\in M), \qquad \Psi=\diag(\Psi_n:n\in M).
    \end{equation*}
    Then $\Phi,\Psi$ are bounded operators on $E(M,X)$, $\norm{\Phi}\norm{\Psi}\leq C$, and
    \begin{equation*}
        \Phi P_M D J_M \Psi=I_{E(M,X)}.
    \end{equation*}
    Identifying $E(M, X)$ with $E(X)$, this shows that $D$ factors the identity on $E(X)$ with constant $C$.

    If instead $I_X-D_n$ factors the identity on $X$ with constant $C$ for every $n\in M$, the same argument applied to $I_{E(X)}-D$ shows that $I_{E(X)}-D$ factors the identity on $E(X)$ with constant $C$.
\end{proof}

\begin{lemma}[Stability of factorisation under perturbations]\label{lmm:factorisation-perturbation}
    Let $X$ be a Banach space and $T, S\in\mathscr{B}(X)$. Suppose that $T$ factors the identity on $X$ with constant $C$ and that
    \begin{equation*}
        C\norm{S-T}<1/2.
    \end{equation*}
    Then $S$ factors the identity on $X$ with constant $2C$.
\end{lemma}

\begin{proof}
    Since $T$ factors the identity with constant $C$, there exist operators $\Phi,\Psi\in\mathscr{B}(X)$ such that
    \begin{equation*}
        \Phi T\Psi=I_X \qquad\text{and}\qquad \norm{\Phi}\norm{\Psi}\leq C.
    \end{equation*}
    Hence
    \begin{equation*}
        \norm{\Phi(S-T)\Psi}\leq C\norm{S-T}<1/2.
    \end{equation*}
    Therefore
    \begin{equation*}
        \Phi S\Psi=I_X+\Phi(S-T)\Psi
    \end{equation*}
    is invertible, and
    \begin{equation*}
        \norm{(\Phi S\Psi)^{-1}}\leq \frac{1}{1-\norm{\Phi(S-T)\Psi}}<2.
    \end{equation*}
    Thus
    \begin{equation*}
        I_X=(\Phi S\Psi)^{-1}\Phi S\Psi.
    \end{equation*}
    Consequently, the identity factors through $S$, and the factorisation constant is at most
    \begin{equation*}
        \norm{(\Phi S\Psi)^{-1}\Phi}\norm{\Psi} \leq \norm{(\Phi S\Psi)^{-1}}\norm{\Phi}\norm{\Psi} <2C.
     \qedhere \end{equation*}
\end{proof}

Equipped with these observations we can finally give a proof for our main results.

\begin{proof}[Proof of \Cref{th: main1}]
    Let $C\geq 1$ be such that $X$ has the $C$-PFP, and choose $\delta>0$ such that $C\delta<1/2$.

    Suppose first that $X^*$ has finite cotype. Let $T\colon \ell_1(X)\to \ell_1(X)$ be an operator. By \Cref{prop:diagonal-reduction-from-type}, there exist an infinite set $M\subseteq\N$ and a diagonal operator $D\colon \ell_1(M,X)\to\ell_1(M,X)$ such that
    \begin{equation*}
        \norm{P_MTJ_M-D}<\delta.
    \end{equation*}
    By \Cref{prop:diagonal-dichotomy}, either $D$ or $I_{\ell_1(M,X)}-D$ factors the identity on $\ell_1(M,X)$ with constant $C$. If $D$ does, then \Cref{lmm:factorisation-perturbation} implies that $P_MTJ_M$ factors the identity with constant $2C$. Hence $T$ factors the identity on $\ell_1(X)$ with constant $2C$. If instead $I_{\ell_1(M,X)}-D$ factors the identity, then, since
    \begin{equation*}
        \norm{P_M(I_{\ell_1(X)}-T)J_M-(I_{\ell_1(M,X)}-D)}<\delta,
    \end{equation*}
    the same argument shows that $I_{\ell_1(X)}-T$ factors the identity on $\ell_1(X)$ with constant $2C$. Thus $\ell_1(X)$ has the UPFP.

    Now suppose that $X\simeq E(X)$, where $E=\ell_p$ for some $1<p\leq\infty$, or $E=c_0$. Since $\ell_1(X)\simeq \ell_1(E(X))$, it is enough to show that $\ell_1(E(X))$ has the UPFP. Let $T\colon \ell_1(E(X))\to \ell_1(E(X))$ be an operator. By \Cref{prop:diagonal-reduction-from-reprod}, there exist increasing sequences $\mathbf{n}=(n_j)_{j\in\N}$ and $\mathbf{m}=(m_j)_{j\in\N}$, and a diagonal operator $D\colon \ell_1(X)\to\ell_1(X)$, such that
    \begin{equation*}
        \norm{P_{\mathbf{n},\mathbf{m}}TJ_{\mathbf{n},\mathbf{m}}-D}<\delta.
    \end{equation*}
    Again, \Cref{prop:diagonal-dichotomy} and \Cref{lmm:factorisation-perturbation} imply that either $P_{\mathbf{n},\mathbf{m}}TJ_{\mathbf{n},\mathbf{m}}$ or $P_{\mathbf{n},\mathbf{m}}(I_{\ell_1(E(X))}-T)J_{\mathbf{n},\mathbf{m}}$ factors the identity on $\ell_1(X)$ with constant $2C$. Hence either $T$ or $I_{\ell_1(E(X))}-T$ factors the identity on $\ell_1(X)$ with constant $2C$. Since $\ell_1(E(X))\simeq\ell_1(X)$, this gives the UPFP for $\ell_1(E(X))$, and therefore for $\ell_1(X)$.

    Finally, if $X\simeq \ell_1(X)$, then $\ell_1(X)$ has the UPFP by isomorphic invariance. The primariness conclusion follows from the UPFP and Pe{\l}czy\'nski's decomposition method, using $\ell_1(X)\simeq \ell_1(\ell_1(X))$.
\end{proof}

\begin{proof}[Proof of \Cref{th: main2}]
    Let $C\geq 1$ be such that $X$ has the $C$-PFP, and choose $\delta>0$ such that $C\delta<1/2$. Let $E=c_0$ or $E=\ell_\infty$.

    Suppose first that $X$ has finite cotype. Let $T\colon E(X)\to E(X)$ be an operator. By \Cref{prop:diagonal-reduction-from-cotype-ellinfty}, there exist an infinite set $M\subseteq\N$ and a diagonal operator $D\colon E(M,X)\to E(M,X)$ such that
    \begin{equation*}
        \norm{P_MTJ_M-D}<\delta.
    \end{equation*}
    By \Cref{prop:diagonal-dichotomy}, either $D$ or $I_{E(M,X)}-D$ factors the identity on $E(M,X)$ with constant $C$. If $D$ does, then \Cref{lmm:factorisation-perturbation} implies that $P_MTJ_M$ factors the identity with constant $2C$. Hence $T$ factors the identity on $E(X)$ with constant $2C$. If instead $I_{E(M,X)}-D$ factors the identity, then, since
    \begin{equation*}
        \norm{P_M(I_{E(X)}-T)J_M-(I_{E(M,X)}-D)}<\delta,
    \end{equation*}
    the same argument shows that $I_{E(X)}-T$ factors the identity on $E(X)$ with constant $2C$. Thus $E(X)$ has the UPFP.

    Now suppose that $X\simeq \ell_p(X)$ for some $1\leq p<\infty$. Since $E(X)\simeq E(\ell_p(X))$, it is enough to show that $E(\ell_p(X))$ has the UPFP.
    
    Let $T\colon E(\ell_p(X))\to E(\ell_p(X))$ be an operator. By \Cref{prop:diagonal-reduction-from-reprod-ellinfty}, there exist increasing sequences $\mathbf{n}=(n_j)_{j\in\N}$ and $\mathbf{m}=(m_j)_{j\in\N}$, and a diagonal operator $D\colon E(X)\to E(X)$, such that
    \begin{equation*}
        \norm{P_{\mathbf{n},\mathbf{m}}TJ_{\mathbf{n},\mathbf{m}}-D}<\delta.
    \end{equation*}
    Again, \Cref{prop:diagonal-dichotomy} and \Cref{lmm:factorisation-perturbation} imply that either $P_{\mathbf{n},\mathbf{m}}TJ_{\mathbf{n},\mathbf{m}}$ or $P_{\mathbf{n},\mathbf{m}}(I_{E(\ell_p(X))}-T)J_{\mathbf{n},\mathbf{m}}$ factors the identity on $E(X)$ with constant $2C$. Hence either $T$ or $I_{E(\ell_p(X))}-T$ factors the identity on $E(X)$ with constant $2C$. Since $E(\ell_p(X))\simeq E(X)$, this gives the UPFP for $E(\ell_p(X))$, and therefore for $E(X)$.

    Since $E$ was arbitrary, both $c_0(X)$ and $\ell_\infty(X)$ have the UPFP. The primariness conclusion follows from the UPFP and Pe{\l}czy\'nski's decomposition method, using $c_0(X)\simeq c_0(c_0(X))$ and $\ell_\infty(X)\simeq \ell_\infty(\ell_\infty(X))$.
\end{proof}
\bigskip
\section{\texorpdfstring{An uncountable $\ell_1$-sum}{An uncountable l1-sum}}\label{sec:uncountable-l1-sum}

Although the preceding transfer principles are stated for countable vector-valued sequence spaces, the diagonal-reduction arguments are not inherently countable. With the evident modifications, the same methods apply to spaces indexed by an arbitrary set $\Gamma$, such as $\ell_1(\Gamma,X)$, $c_0(\Gamma,X)$ and $\ell_\infty(\Gamma,X)$. The only caveat is that for $X \simeq \ell_1(X)$, the UPFP of $X$ does not automatically transfer to the UPFP of $\ell_1(\Gamma, X)$ when $\Gamma$ is uncountable.

The uncountable setting allows an additional diagonal-reduction mechanism. In the countable arguments above, both upper- and lower-triangular reductions are needed in order to reach a diagonal operator. For uncountable index sets, one can sometimes replace this two-sided selection by a one-sided finite-interference estimate, followed by a free-set argument which removes the finitely many exceptional coordinates on a large subset of the index set.

We now apply this mechanism to the summand $L_1[0,1]$. Although $L_1[0,1]$ is not covered by the finite-dual-cotype hypothesis in \Cref{th: main1}, it has enough additional measure-theoretic structure to recover the required diagonal reduction. Capon's localisation argument gives the UPFP for $L_1[0,1]$, while Kalton's representation theorem gives the finite-interference estimate for operators from $L_1[0,1]$ into $\ell_1(\Gamma,L_1[0,1])$. Combining this estimate with Hajnal's free set theorem yields the diagonal reduction for every uncountable $\Gamma$. Together with the countable case, this proves \Cref{th:main3}. The application to $C[0,1]^*$ then follows from its classical representation as an $\ell_1$-sum of copies of $L_1[0,1]$.

We begin by recalling the localisation argument of Capon. If $B\subseteq[0,1]$ is measurable, we write
\begin{equation*}
    L_1(B)=\{f\in L_1[0,1]: f=0 \text{ a.e. on } [0,1]\setminus B\},
\end{equation*}
and we denote by $M_B \colon L_1[0,1] \to L_1(B)$ the multiplication operator $f\mapsto \mathbf{1}_B f$ and by $\jmath_B \colon L_1(B) \to L_1[0,1]$ the canonical inclusion. Thus, when $R\colon L_1[0,1]\to L_1[0,1]$ is an operator, $M_B R$ denotes the map $f\mapsto \mathbf{1}_B Rf$. Finally, for any operator $S \colon L_1[0,1] \to L_1[0,1]$ we can associate to it a family of measures $(\nu_t)_{t \in [0,1]}$ according to Kalton's representation theorem \cite[Theorem 3.1 (b)]{Kalton1978}.

\begin{lemma}[Capon localisation]\label{lem:capon-localisation}
Let $S\colon L_1[0,1]\to L_1[0,1]$ be an operator, and let $(\nu_t)_{t\in[0,1]}$ be the family of measures associated with $S$ by Kalton's representation theorem. Suppose that, for some $a>0$, the set
\begin{equation*}
    E_a(S)=\{t\in[0,1]:|\nu_t(\{t\})|>a\}
\end{equation*}
has positive measure. Then, for every $0<\varepsilon<a$, there is a measurable set $A\subseteq[0,1]$ of positive measure such that, whenever $B\subseteq A$ has positive measure, the operator
\begin{equation*}
    M_BS\jmath_B\colon L_1(B)\to L_1(B)
\end{equation*}
is an isomorphism onto $L_1(B)$, and
\begin{equation*}
    \norm{M_BS\jmath_B f}\geq (a-\varepsilon)\norm{f}\qquad(f\in L_1(B)).
\end{equation*}
Consequently,
\begin{equation*}
    \norm{(M_BS\jmath_B)^{-1}}\leq\frac{1}{a-\varepsilon}.
\end{equation*}
\end{lemma}

\begin{proof}
    This is Capon's localisation argument \cite[Proposition~1.2]{Capon1980}. Capon stated the result for the value $a=1/2$; the proof is unchanged with $1/2$ replaced by an arbitrary $a>0$.
\end{proof}

As an immediate consequence, we obtain the following.

\begin{corollary}\label{cor:L1-has-UPFP}
    The space $L_1[0,1]$ has the UPFP.
\end{corollary}

\begin{proof}
    Let $T\colon L_1[0,1]\to L_1[0,1]$ be an operator, and let $(\nu_t)_{t\in[0,1]}$ be the associated family of measures. The operator $I_{L_1[0,1]}-T$ is associated with the family $(\delta_t-\nu_t)_{t\in[0,1]}$. For every $t\in[0,1]$, either
    \begin{equation*}
        |\nu_t(\{t\})|>\frac13
    \end{equation*}
    or
    \begin{equation*}
        |(\delta_t-\nu_t)(\{t\})|>\frac13.
    \end{equation*}
    Hence one of these two alternatives holds on a set of positive measure.

    Applying \Cref{lem:capon-localisation} with $a=1/3$ and $\varepsilon=1/6$, either to $T$ or to $I_{L_1[0,1]}-T$, we obtain a positive-measure set $B\subseteq[0,1]$ such that the corresponding operator
    \begin{equation*}
        M_B T \jmath_B \colon L_1(B)\to L_1(B)
    \end{equation*}
    or
    \begin{equation*}
        M_B(I_{L_1[0,1]}-T)\jmath_B\colon L_1(B)\to L_1(B)
    \end{equation*}
    is an isomorphism onto $L_1(B)$, whose inverse has norm at most $6$. Since $L_1(B)$ is isometric to $L_1[0,1]$, the identity on $L_1[0,1]$ factors through either $T$ or $I_{L_1[0,1]}-T$ with a constant independent of $T$. Therefore, $L_1[0,1]$ has the UPFP.
\end{proof}

The next lemma is the one-sided localisation step used in the uncountable diagonal reduction. It is analogous in spirit to the forward reductions from \Cref{sec:upper-triangular}: after passing to a suitable positive-measure part of the domain, all but finitely many target coordinates of the operator become small. In the uncountable setting, this finite exceptional set can later be removed simultaneously along a large set of coordinates by Hajnal's free set theorem.

\begin{lemma}[Finite interference]\label{lem:finite-capture}
Let $\Lambda$ be a set and $S\colon L_1[0,1]\to \ell_1(\Lambda,L_1[0,1])$ be an operator. Then, for every $\eta>0$, there are a finite set $F\subseteq\Lambda$ and a measurable set $B\subseteq[0,1]$ of positive measure such that
\begin{equation*}
    \norm{P_{\Lambda\setminus F}S \jmath_B}<\eta.
\end{equation*}
\end{lemma}

\begin{proof}
We first prove the result when $\Lambda$ is countable and thus we may assume without loss of generality that $\Lambda=\N$. Identify $\ell_1(\N,L_1[0,1])$ with $L_1(\Omega,m)$, where $\Omega=\N\times[0,1]$ and $m$ is the product of the counting measure on $\N$ and the Lebesgue measure on $[0,1]$. Let $\mu$ denote Lebesgue measure on $[0,1]$. By Kalton's representation theorem \cite[Theorem 3.1(b)]{Kalton1978}, there is a measurable family $(\nu_y)_{y\in\Omega}$ of finite measures on $[0,1]$ such that
\begin{equation*}
    (Sf)(y)=\int_{[0,1]} f\,d\nu_y
\end{equation*}
for almost every $y\in\Omega$.

For $N\in\N$, put $F_N=\{1,\ldots,N\}$ and $\Omega_N=(\N\setminus F_N)\times[0,1]$. Define a finite positive measure $\rho_N$ on $[0,1]$ by
\begin{equation*}
    \rho_N(E)=\int_{\Omega_N}|\nu_y|(E)\,dm(y).
\end{equation*}
By Kalton's representation theorem \cite[Theorem~3.1(b)]{Kalton1978}, applied to $P_{\N\setminus F_N}S$, and using the estimate \cite[(3.1.1)]{Kalton1978}, we have
\begin{equation*}
    \rho_N(E)\leq\norm{P_{\N\setminus F_N}S}\mu(E)\leq\norm{S}\mu(E)
\end{equation*}
for every measurable $E\subseteq[0,1]$. Hence $\rho_N$ is absolutely continuous with respect to $\mu$.

Let $g_N$ be the Radon--Nikodym derivative of $\rho_N$ with respect to $\mu$. Since $g_N$ is finite almost everywhere, the set
\begin{equation*}
    B=\{t\in[0,1]:g_N(t)<\eta/2\}
\end{equation*}
has positive measure for some $N\in\N$. Indeed, otherwise $g_N\geq\eta/2$ almost everywhere for every $N$, and hence
\begin{equation*}
    \rho_N([0,1])=\int_0^1 g_N\,d\mu\geq\eta/2\qquad(N\in\N),
\end{equation*}
which contradicts that $\rho_N([0,1])\to 0$ as $N \to \infty$. The latter convergence follows from $\Omega_N\downarrow\varnothing$ and the integrability of $y\mapsto|\nu_y|([0,1])$ over $\Omega$.

It remains to estimate the norm of $P_{\N\setminus F_N}S \jmath_B$. Let $f\in L_1(B)$. Then, for almost all $y\in\Omega_N$, we have
\begin{equation*}
    (P_{\N\setminus F_N}S \jmath_B f)(y)=\int_B f\,d\nu_y.
\end{equation*}
Hence, by Tonelli's theorem and the definition of $\rho_N$,
\begin{equation*}
    \norm{P_{\N\setminus F_N}S \jmath_B f}\leq\int_{\Omega_N}\int_B |f|\,d|\nu_y|\,dm(y)=\int_B |f|\,d\rho_N.
\end{equation*}
Since $d\rho_N=g_N\,d\mu$ and $g_N<\eta/2$ on $B$, we get
\begin{equation*}
    \int_B |f|\,d\rho_N=\int_B |f|g_N\,d\mu\leq(\eta/2)\norm{f}.
\end{equation*}
Therefore
\begin{equation*}
    \norm{P_{\N\setminus F_N}S \jmath_B}<\eta.
\end{equation*}

Now let $\Lambda$ be arbitrary. Since $L_1[0,1]$ is separable, choose a dense sequence $(f_n)_{n\in\N}$ in $L_1[0,1]$. For each $n$, the vector $Sf_n$ has countable support in $\Lambda$, so
\begin{equation*}
    \Lambda_0=\bigcup_{n=1}^\infty\supp(Sf_n)
\end{equation*}
is countable. If $\lambda\in\Lambda\setminus\Lambda_0$, then $P_\lambda Sf_n=0$ for every $n$, and hence $P_\lambda S=0$ by continuity. Thus
\begin{equation*}
    S(L_1[0,1])\subseteq \ell_1(\Lambda_0,L_1[0,1]).
\end{equation*}
Applying the countable case to $S$ regarded as an operator into $\ell_1(\Lambda_0,L_1[0,1])$, we obtain a finite set $F\subseteq\Lambda_0$ and a measurable set $B\subseteq[0,1]$ of positive measure such that
\begin{equation*}
    \norm{P_{\Lambda_0\setminus F}S \jmath_B}<\eta.
\end{equation*}
Since $S$ has no coordinates outside $\Lambda_0$, this is the same as
\begin{equation*}
    \norm{P_{\Lambda\setminus F}S \jmath_B}<\eta.
\end{equation*}
The proof is complete.
\end{proof}

Before we proceed to the diagonal reduction, we recall the following finite-valued form of Hajnal's free set theorem \cite{Hajnal1961}.

\begin{lemma}[Finite free set]\label{lem:free}
Let $\Gamma$ be an uncountable set, and let $F_\gamma\subseteq\Gamma\setminus\{\gamma\}$ be finite for each $\gamma\in\Gamma$. Then there is a set $H\subseteq\Gamma$ such that $|H|=|\Gamma|$ and, for every $\gamma \in H$, we have
\begin{equation*}
    F_\gamma\cap H=\varnothing.
\end{equation*}
\end{lemma}

We are now ready to obtain the analogue, in the present uncountable setting, of the diagonal reductions proved in \Cref{sec:diagonal-and-proofs}. The key point is that the finite interference lemma gives a one-sided reduction for each column, while Hajnal's free set theorem removes the finitely many exceptional coordinates simultaneously on a large subset of $\Gamma$.

\begin{lemma}[Diagonal reduction]\label{lem:uncountable-diagonal-reduction}
Let $\Gamma$ be an uncountable set, and let $T\colon \ell_1(\Gamma,L_1[0,1])\to \ell_1(\Gamma,L_1[0,1])$ be an operator. Then, for every $\varepsilon>0$, there are a set $H\subseteq\Gamma$ with $|H|=|\Gamma|$, measurable sets $B_\gamma\subseteq[0,1]$ of positive measure for $\gamma\in H$, contractions
\begin{equation*}
    V\colon \left(\bigoplus_{\gamma\in H}L_1(B_\gamma)\right)_{\ell_1}\to \ell_1(H,L_1[0,1])
\end{equation*}
and
\begin{equation*}
    U\colon \ell_1(H,L_1[0,1])\to \left(\bigoplus_{\gamma\in H}L_1(B_\gamma)\right)_{\ell_1}
\end{equation*}
such that
\begin{equation*}
    UV=I_{\bigl(\bigoplus_{\gamma\in H}L_1(B_\gamma)\bigr)_{\ell_1}},
\end{equation*}
and a diagonal operator $D$ on $\bigl(\bigoplus_{\gamma\in H}L_1(B_\gamma)\bigr)_{\ell_1}$ such that
\begin{equation*}
    \norm{U P_H T J_H V-D}\leq\varepsilon.
\end{equation*}
\end{lemma}

\begin{proof}
For each $\gamma\in\Gamma$, define
\begin{equation*}
    S_\gamma=P_{\Gamma\setminus\{\gamma\}}TJ_\gamma\colon L_1[0,1]\to\ell_1(\Gamma\setminus\{\gamma\},L_1[0,1]).
\end{equation*}
Apply \Cref{lem:finite-capture} to $S_\gamma$, with $\eta = \varepsilon$. We obtain a finite set $F_\gamma\subseteq\Gamma\setminus\{\gamma\}$ and a measurable set $B_\gamma\subseteq[0,1]$ of positive measure such that
\begin{equation*}
    \norm{P_{\Gamma\setminus(\{\gamma\}\cup F_\gamma)}TJ_\gamma \jmath_{B_\gamma}}<\varepsilon.
\end{equation*}
By \Cref{lem:free}, there is a set $H\subseteq\Gamma$ with $|H|=|\Gamma|$ such that, for all $\gamma \in H$, we have
\begin{equation*}
    F_\gamma\cap H=\varnothing.
\end{equation*}
Hence, for every $\gamma\in H$,
\begin{equation*}
    \norm{P_{H\setminus\{\gamma\}}TJ_\gamma \jmath_{B_\gamma}}<\varepsilon.
\end{equation*}

Put
\begin{equation*}
    Y=\left(\bigoplus_{\gamma\in H}L_1(B_\gamma)\right)_{\ell_1}.
\end{equation*}
Define $V\colon Y\to\ell_1(H,L_1[0,1])$ and $U\colon\ell_1(H,L_1[0,1])\to Y$ coordinatewise, for each $\gamma\in H$, by
\begin{equation*}
    (Vy)_\gamma=\jmath_{B_\gamma}y_\gamma,\qquad (Ux)_\gamma=M_{B_\gamma}x_\gamma.
\end{equation*}
Then $U$ and $V$ are contractions, and $UV=I_Y$.

Let $D\colon Y\to Y$ be the diagonal operator given, for each $\gamma\in H$, by
\begin{equation*}
    (Dy)_\gamma=M_{B_\gamma}P_\gamma TJ_\gamma \jmath_{B_\gamma}y_\gamma.
\end{equation*}
Estimating by columns and using $F_\beta\cap H=\varnothing$, we obtain, for $y=(y_\beta)_{\beta\in H}\in Y$,
\begin{equation*}
    \norm{(UP_HTJ_HV-D)y}\leq\sum_{\beta\in H}\norm{P_{H\setminus\{\beta\}}TJ_\beta \jmath_{B_\beta}y_\beta}\leq\varepsilon\norm{y}.
\end{equation*}
Therefore
\begin{equation*}
    \norm{U P_H T J_H V-D}\leq\varepsilon.
\end{equation*}
\end{proof}

Once the operator has been reduced to diagonal form up to a small perturbation, the factorisation argument is the same as in the proof of the main transfer theorems. We first record the corresponding diagonal dichotomy for the localised $\ell_1$-sum.

\begin{proposition}[Diagonal dichotomy]\label{prop:uncountable-diagonal-dichotomy}
Let $C\geq1$, let $H$ be an uncountable set, and for each $\alpha\in H$ let $B_\alpha\subseteq[0,1]$ be measurable with positive measure. Suppose that $L_1[0,1]$ has the $C$-PFP. Put
\begin{equation*}
    Y=\left(\bigoplus_{\alpha\in H}L_1(B_\alpha)\right)_{\ell_1}.
\end{equation*}
Then, for every diagonal operator $D\colon Y\to Y$, either $D$ or $I_Y-D$ factors the identity on $Y$ with constant $C$.
\end{proposition}

\begin{proof}
The proof is identical to that of \Cref{prop:diagonal-dichotomy}, using that each $L_1(B_\alpha)$ is isometric to $L_1[0,1]$.
\end{proof}

We are finally ready for the proof of \Cref{th:main3}.

\begin{proof}[Proof of \Cref{th:main3}]
If $\Gamma$ is finite or countably infinite, the result follows from \Cref{cor:L1-has-UPFP}, since $\ell_1(\Gamma,L_1[0,1])$ is isomorphic to $L_1[0,1]$. We may therefore suppose that $\Gamma$ is uncountable.

Let $C\geq1$ be such that $L_1[0,1]$ has the $C$-PFP, and choose $\delta>0$ such that $C\delta<1/2$. Let $T\colon \ell_1(\Gamma,L_1[0,1])\to \ell_1(\Gamma,L_1[0,1])$ be an operator.

Applying \Cref{lem:uncountable-diagonal-reduction} with $\varepsilon = \delta/2$, we obtain a set $H\subseteq\Gamma$ with $|H|=|\Gamma|$, measurable sets $B_\gamma\subseteq[0,1]$ of positive measure for $\gamma \in H$, contractions
\begin{equation*}
    V\colon Y\to \ell_1(H,L_1[0,1]),\qquad U\colon \ell_1(H,L_1[0,1])\to Y,
\end{equation*}
where
\begin{equation*}
    Y=\left(\bigoplus_{\gamma\in H}L_1(B_\gamma)\right)_{\ell_1},
\end{equation*}
such that $UV=I_Y$, and a diagonal operator $D\colon Y\to Y$ satisfying
\begin{equation*}
    \norm{U P_H T J_H V-D} \leq \delta/2 < \delta.
\end{equation*}
By \Cref{prop:uncountable-diagonal-dichotomy}, either $D$ or $I_Y-D$ factors the identity on $Y$ with constant $C$.

Suppose first that $D$ factors the identity on $Y$ with constant $C$. Since
\begin{equation*}
    C\norm{U P_H T J_H V-D}<1/2,
\end{equation*}
\Cref{lmm:factorisation-perturbation} implies that $U P_H T J_H V$ factors the identity on $Y$ with constant $2C$. Choose an isometric isomorphism
\begin{equation*}
    W\colon \ell_1(\Gamma,L_1[0,1])\to Y,
\end{equation*}
which exists because $|H|=|\Gamma|$ and each $L_1(B_\gamma)$ is isometric to $L_1[0,1]$. Since $W$ and $W^{-1}$ are isometries, it follows that $T$ factors the identity on $\ell_1(\Gamma,L_1[0,1])$ with constant $2C$.

Suppose instead that $I_Y-D$ factors the identity on $Y$ with constant $C$. Since $UV=I_Y$, we have
\begin{equation*}
    I_Y-U P_H T J_H V=U P_H(I_{\ell_1(\Gamma,L_1[0,1])}-T)J_H V.
\end{equation*}
Moreover,
\begin{equation*}
    \norm{U P_H(I_{\ell_1(\Gamma,L_1[0,1])}-T)J_H V-(I_Y-D)}<\delta.
\end{equation*}
Thus, arguing as before we conclude that $I_{\ell_1(\Gamma,L_1[0,1])} - T$ factors the identity with constant $2C$, which proves that $\ell_1(\Gamma, L_1[0,1])$ has the UPFP.

Lastly, by a classical result \cite[Proposition~5.6]{LT2}, the space $C[0,1]^*$ is isomorphic to $\ell_1(2^{\aleph_0},L_1[0,1])$. Hence, by the previous argument, $C[0,1]^*$ has the UPFP, and therefore $C[0,1]^*$ is primary.
\end{proof}
\bigskip
\section{Final discussion and open questions}\label{sec:remarks-and-questions}

The spaces $\ell_1$, $c_0$ and $\ell_\infty$, especially the latter, have historically been among the more delicate settings for factorisation and primariness arguments, often requiring arguments different from those used for $\ell_p$, $1<p<\infty$, or even for sums with respect to arbitrary symmetric bases. Somewhat surprisingly, the methods presented here can only be directly applied in these cases. The reason is that the geometry of $\ell_1$ and of the supremum norm gives enough control over coordinate restrictions to obtain both upper- and lower-triangular reductions.

For $\ell_p$-sums with $1<p<\infty$, the situation is less clear. One can often obtain one-sided reductions, either upper or lower triangular, under natural hypotheses, but obtaining both simultaneously seems to require substantially stronger assumptions. This obstruction is not absolute: in \cite{Acuaviva2026PrimarinesslpCK}, a two-sided reduction was obtained for $\ell_p$-sums of certain $C(K)$ spaces. There, the upper-triangular reduction used the $c_0$-structure of the spaces, while the lower-triangular reduction relied on the special structure of $C(K)$, in particular on reproducibility properties of $K$ which force an $\ell_1$-type behaviour.

It is natural to ask whether there are general hypotheses under which analogous two-sided reductions can be proved for $\ell_p(X)$, $1<p<\infty$.

\begin{question}
    Let $1<p<\infty$. Are there natural and sufficiently broad conditions on a Banach space $X$ under which the UPFP of $X$ passes to $\ell_p(X)$?
\end{question}

\begin{remark}
    In the uncountable setting, the need for a full two-sided reduction can sometimes be avoided. Indeed, once a one-sided finite-interference estimate is available, Hajnal's free set theorem can remove the finitely many exceptional coordinates on a large subset of the index set. Thus the cotype and self-similarity mechanisms used in the countable reductions above are expected to give corresponding UPFP results for uncountable sums $\ell_p(\Gamma,X)$, $1<p<\infty$, under the analogous hypotheses.
\end{remark}

The PFP and UPFP originate in the study of primariness, so it is natural to ask whether these properties, or primariness itself, are preserved by vector-valued sequence-space constructions.

\begin{question}
    Is there a Banach space $X$ with the PFP, respectively the UPFP, such that $\ell_p(X)$, for some $1\leq p\leq\infty$, or $c_0(X)$ fails the PFP, respectively the UPFP?
\end{question}

\begin{question}
    Is there a primary Banach space $X$ such that $\ell_p(X)$, for some $1\leq p\leq\infty$, or $c_0(X)$ is not primary?
\end{question}

A counterexample to either question would likely have to be rather pathological, while a positive answer seems far beyond the reach of the present methods.

\bigskip

\noindent\textbf{Acknowledgements.} This paper forms part of the author’s PhD research at Lancaster University, conducted under the supervision of Professor N. J. Laustsen. The author is deeply grateful to Professor Laustsen for his insightful comments and valuable suggestions, which significantly improved the manuscript's presentation. He acknowledges with thanks the funding from the EPSRC (grant number EP/W524438/1) that has supported his studies. \medskip

For the purpose of open access, the author has applied a Creative Commons Attribution (CC BY) licence to any Author Accepted Manuscript version arising. \medskip

\noindent\textbf{Data availability.} No data was used for the research described in the article.


\end{document}